\documentclass[10pt]{cmslatex}

\usepackage[paperwidth=7in, paperheight=10in,margin=.875in]{geometry}

\usepackage{amsfonts,amssymb,mathtools}
\usepackage{color}

\newcommand{\Rb}{\mathbb R}
\newcommand{\Cc}{\mathcal C}
\newcommand{\Mc}{\mathcal M}
\newcommand{\Uc}{\mathcal U}

\newcommand{\veps}{\varepsilon}
\newcommand{\vphi}{\varphi}

\newcommand{\indic}[1]{\mathbf{1}_{#1}}
\newcommand{\ds}{\displaystyle}

\newcommand{\ie}{{\it i.e.}}

\newtheorem{thm}{Theorem}
\newtheorem{propo}{Proposition}
\newtheorem{lem}{Lemma}

\newtheorem{defn}{Definition}

\newtheorem{hyp}{Assumption}
\newtheorem{remk}{Remark}
        
\begin{document}

\title{Quasi steady state approximation of the small clusters in Becker-D\"oring  equations leads to boundary 
conditions in the Lifshitz-Slyozov limit}

\author{{Julien Deschamps\thanks{DIMA, Universit\`a degli Studi di Genova, Italy (deschamps@dima.unige.it)}}\and{Erwan 
Hingant\thanks{Departamento~de~Matem\'atica,~Universidade~Federal~de~Campina~Grande, Paraiba, ~Brasil 
(erwan@mat.ufcg.edu.br).}}\and{Romain Yvinec\thanks{Physiologie de la Reproduction et des Comportements, Institut 
National de la Recherche Agronomique (INRA) UMR85, CNRS-Universit\'e Fran\c{c}ois-Rabelais UMR7247, IFCE, Nouzilly, 
37380 France (romain.yvinec@tours.inra.fr)} }}

\thanks{EH thanks the financial support of CAPES/IMPA (Brazil)}


\pagestyle{myheadings} \markboth{Becker-D\"oring  equations to the Lifshitz-Slyozov limit with boundaries}{J. 
Deschamps, E. Hingant and R. Yvinec} 

\maketitle

\begin{abstract}
This papers addresses the connection between two classical models of phase transition phenomena 
describing different stages of the growth of clusters. The Becker-D\"oring model 
(BD) describes discrete-sized clusters through an infinite set of ordinary differential equations. The Lifshitz-Slyozov equation (LS) is a 
transport partial differential equation on the continuous half-line $x\in (0,+\infty)$. We introduce a scaling parameter $\veps>0$, which accounts for the grid size of the state space in the BD model, and recover the LS model in the limit $\veps\to 0$. 
The connection has been already proven in the context of outgoing characteristic at the boundary $x=0$ for the LS model, when small clusters tend to shrink. The main novelty of this work resides in a new estimate on the growth of small clusters, which behave at a fast time scale.
Through a rigorous quasi steady state approximation, we derive boundary conditions for the incoming characteristic case, when small clusters tend to grow.

\end{abstract}



\section{Introduction}

This papers addresses the mathematical connection between two classical models of phase transition phenomena 
describing different stages of the growth of clusters (or polymers, or aggregates). The first one is the 
Becker-D\"oring model (BD), first introduced in \cite{BD}, that describes the earlier stages of cluster growth, 
at a small scale. Cluster of particles may increase or decrease their size one-by-one, capturing (aggregation process) 
or shedding (fragmentation process) one particle, according to the set of chemical reactions 
\begin{equation*}
  \ds C_1 + C_i \ds \xrightleftharpoons  \ds C_{i+1} \, \quad i\geq 1\,,
\end{equation*}
where $C_i$ stands for the clusters consisting of $i$ particles, $C_1$ being the single {\it free} particle. In 
its mean-field version, the BD model is an infinite set of ordinary differential equations for the time 
evolution of each concentrations (numbers per unit of volume) of clusters made of $i$ particles. In this works we focus 
on a dimensionless BD model that involves a small parameter $\veps>0$. The standard scaling procedure is detailed in 
Appendix \ref{annex:adimensionalixation}. We denote by $c^\veps_i(t)$ the concentration at time $t\geq 0$ of 
clusters consisting of $i\geq 2$ particles and $u^\veps$ for the concentration of {\it free} particles (clusters  of 
size $1$), where we make explicit the dependence on $\veps>0$. The dimensionless system reads 
\begin{equation}\label{sys:BD_rescaled}
 \begin{array}{rclr}
 \ds \frac{d}{dt} u^\veps & = & \ds -  \veps  J_{1}^\veps - \veps  \sum_{i \geq {1}} J_i^\veps \,, & t\geq 0\,,\\[0.8em]
 \ds \frac{d}{dt} c_{i}^\veps & = & \ds \frac{1}{\veps} \Big{[} J_{i-1}^\veps -  J_i^\veps \Big{]}\,, & i\geq 2 
\,, \ t\geq 0\,, \\[0.8em]
\end{array}
\end{equation}
where the fluxes are defined by
\begin{equation}\label{sys:BD_rescaled_flux}
 J_{1}^\veps =  \alpha^\veps (u^\veps)^2- \veps^{\eta} \beta^\veps c_{2}^\veps\,,\ \text{ and } \ J_{i}^\veps = 
a_i^\veps u^\veps c_i^\veps - b_{i+1}^\veps c_{i+1}^\veps\,, \quad i\geq 2\,.
\end{equation}
Here, coefficients $a_i^\veps$ and $b_{i+1}^\veps$, for $i\geq 2$, denote respectively the rate of aggregation and 
fragmentation ($\veps$-dependent), while $\alpha^\veps$ and $\beta^\veps$ denote respectively the 
first rate of aggregation ($i=1$) and the first rate of fragmentation ($i=2$). Finally, $\eta$ is an exponent that 
stands for the strength of the first fragmentation rate, on which strongly depends our results (see also Section 
\ref{sec:disc} for discussions). Observe that such model (at least formally) preserves the total number of particles 
(no source nor sink), that is
\begin{equation}\label{eq:mass}
u^\veps(t) + \sum_{i\geq 2} \veps^2  i c_i^\veps(t) = m^\veps  \,, \quad \forall t\geq 0\,.
\end{equation}
The constant $m^\veps$ is entirely determined by the initial conditions at $t=0$ given by $u^{\rm in, \veps}$ and 
$(c^{\rm in,\veps}_i)_{i\geq 2}$, non-negatives and $\veps$-dependent.  
For theoretical studies on the well-posedness and long-time behaviour of the deterministic Becker-D\"oring model (with 
$\veps=1$), we 
refer the interested reader to \cite{Penrose2001,Wattis2008,Laurencot2002a} among many others.

\smallskip

The second model of phase transition is the Lifshitz-Slyozov model (LS) introduced in \cite{LS}. It classically 
describes the late phase of cluster growth, at a ``macroscopic scale''. The LS model consists in a partial 
differential equation (of nonlinear transport type) for the time evolution of the size distribution function $f(t,x)$ 
of clusters of (continuous) size $x>0$ at time $t\geq 0$, together with an equation stating the conservation of matter,

\begin{equation}\label{eq:LS}
 \begin{array}{ll}
  \ds \frac{\partial f}{\partial t} + \frac{\partial [ (a(x)u(t) - b(x))f(t,x) ] }{\partial x}=0 \,, & \quad t\geq0\,, \ x>0\,, \\[1.5em] 
  \ds u(t)  +  \int_0^\infty xf(t,x) = m \, , & \quad t\geq 0\,, 
 \end{array}
\end{equation}
where $a$ and $b$ are functions of the size, respectively for the aggregation and fragmentation rates. The constant $m$ 
plays the same role as in the BD model. Various authors studied this equation when the flux point outward at $x=0$ 
(when small clusters tends to fragment), namely if a condition like $a(0)u(t) - b(0) < 0$ holds, see   
\cite{Laurenccot2002,Collet2002a,Niethammer2008} among other for theoretical studies and technical assumptions. Indeed, 
in that case, uniqueness of weak solution to the limit system \eqref{eq:LS} holds. But, recent applications in biology 
have raised this problem to include {\it nucleation} in this equation (small clusters tends to aggregate), for instance 
in \cite{Prigent2012,Helal2013,Banks2014}. These cases consider fluxes that point inward at $x=0$, and it lacks a 
\textit{boundary condition} to \eqref{eq:LS} to be well-defined. Remark, some boundary conditions was conjectured {\it 
e.g.} in \cite{Collet2002,COLLET2004,Prigent2012} but never rigorously proved. 

\medskip

In this works we aim to recover a solution of the LS equation and construct proper boundary condition, 
departing from the BD equation \eqref{sys:BD_rescaled} as the parameter $\veps$ goes to $0$. This connection has been 
proved in \cite{Collet2002,Laurencot2002a} for the classical case of outgoing characteristic. The authors represent the 
dynamics of the BD model by a  density function on a continuous size space. Accordingly, the 
size of the clusters are represented by a continuous variable $x>0$, and we let, for all $\veps>0$,
\begin{equation} \label{eq:def_feps}
 f^\veps(t,x) = \sum_{i\geq 2} c_i^\veps(t)\indic{\Lambda_i^\veps}(x)\,, \quad x\geq 0\,, \ t\geq 0\,,
\end{equation}
where for each $i\geq 2$, $\Lambda_i^\veps=[(i-1/2)\veps, (i+1/2)\veps)$. We denote for the remainder $f^{\rm in,\, 
\veps}:= f^\veps(0,x)$.  Hence, each cluster of (discrete) size 
initially $i\geq2$ is seen as a cluster of size roughly $i\veps\in\Rb_+$. This scaling consists in an acceleration of 
the fluxes (by $1/\veps$) in Eq.~\eqref{sys:BD_rescaled} so that it can reach an (asymptotically) {\it infinite} size $i = x/\veps$ in 
finite time. Then, an appropriate scaling of the initial conditions, with a large excess of particles, together with 
the rate functions entails that $\{f^\veps\}$ converges to a solution of the LS model, Eq. \eqref{eq:LS}. Here we use the 
same strategy to construct solutions to \eqref{eq:LS} and we derive appropriate flux conditions at $x=0$  when 
the reaction rates behave near $0$ as a power-law, that is
\begin{equation*}
 a(x) \sim_{0^+} \overline{a}x^{r_a}\, \text{ and }\ b(x) \sim_{0^+} \overline{b}x^{r_b}\,,
\end{equation*}
with $\overline{a}$ and $\overline{b}$ positives, and the exponents $0\leq r_a <1$, $r_a \leq r_b$ which corresponds to 
entrant characteristic. Note, if $r_a=r_b$ we suppose moreover that $u(t)>\overline{b}/\overline{a}$. 

\medskip

\begin{remk}
Another scaling approach considers the large time behavior of the Becker-D\"oring model, 
and relates the dynamics of large clusters to solutions of various version of Lifshitz-Slyozov equations. It is the 
so-called theory of Ostwald ripening, see \cite{Penrose1997,Niethammer2005,Velazquez1998}.
\end{remk}

\medskip 

We emphasize that the novelty of our work resides in the rigorous derivation of a boundary condition at $x=0$ for the LS model, Eq. \eqref{eq:LS}, which is needed in the case of entrant characteristic. Thanks to new estimates on the BD model (Proposition \ref{prop:bound_laplace}), we identify the
 limit of quantities related to the (finite size) $c_i^\veps$'s by a quasi steady state approximation. From this, we were able to 
found various possible boundary conditions depending on different scaling hypotheses on the first fragmentation rate, 
{\it i.e.} according to the value of $\eta$ in \eqref{sys:BD_rescaled_flux}, with respect to $r_a$ and $r_b$. Namely, 
we found three distinct cases for {\it slow} de-nucleation rate ($\eta>r_a$) in Theorem \ref{thm:LS_slow}, {\it compensated} one ($\eta=r_a$) 
in Theorem \ref{thm:LS_compensated} and {\it fast} one ($\eta<r_a$) in Theorem \ref{thm:LS_fast}. We 
obtained these main results for measure-valued solution to the LS equation, in Section \ref{sec:results}. But in 
Section \ref{sec:density}, we improve this result to obtain density solution when  $a$ and $b$ are exact power law. Let us give an example of our result to illustrate it. 

\medskip

\noindent \textsc{Illustrating result.} \textit{Assume, for all $x\geq 0$,  $a(x) = \overline a x^{r_a}$ and 
$b(x)=\overline b 
x^{r_b}$ with $r_a < r_b$ and  $\eta = 
r_b$. We found the limit of $\{f^\veps\}$ is a solution of Eq. \eqref{eq:LS}, with the boundary value given by, for 
all 
$t\geq 0$ where $u(t)>0$,
\[ \lim_{x\to 0^+} (a(x)u(t)-b(x))f(t,x) = \alpha u(t)^2\,, \]
where $\alpha$ is the limit of $\alpha^\veps$ in \eqref{sys:BD_rescaled_flux}. In other terms we recover the behavior of $f$ near $x=0$ with the free 
particles concentration through the limit
\[ \lim_{x\to 0^+} x^{r_a} f(t,x) = \frac{\alpha}{\overline a} u(t)\,. \]
}

\medskip

\noindent {\it Organization of the paper.} In the next Section \ref{sec:scaling} we introduce the main assumptions 
together with some properties of the BD model. Then, in Section \ref{sec:results} we state our main result on 
measure-valued solution to LS with boundary term. To do so we improved previous compactness arguments on the re-scaled 
density \eqref{eq:def_feps}, so that the boundary term can be taken into account in Section \ref{sec:compact}. It is 
achieved thanks to a new estimate on the growth of the ``small'' sized clusters (point-wise estimates of the density 
approximation, see Proposition \ref{prop:bound_laplace}). The identification of the boundary term in Section 
\ref{sec:identif} follows from a rigorous quasi-steady-state approximation of the small-sized clusters, in analogy with 
slow-fast systems, and allow proving the main theorems. Finally, we extend some results to a convergence in density, 
see Section \ref{sec:density}. We conclude by a discussion and further directions in Section \ref{sec:disc}.

\medskip

\noindent \textit{Notations.} For any $U\subseteq \Rb$, we denote by $\Cc(U)$, respectively $\Cc_c(U)$ and $\Cc_b(U)$, 
the space of continuous function on $U$, respectively with compact support on $U$, and  bounded on $U$. We 
denote by $\Mc_f(U)$ the set of non-negative and finite 
regular Borel measures on $U$. We will 
use the classical $weak-*$ convergence (sometimes called {\it vague}) on $\Mc_f(U)$, \ie~ the topology given by 
pointwise convergence for test functions  $\vphi\in \Cc_c(U)$, {\it i.e.} for $\{\nu^\veps\}$ and $\nu$ in $\Mc_f(U)$, 
we say $\nu^\veps$ converge to $\nu$ in $\Mc_f(U)$ (in the $weak-*$ topology) if and only if for all $\vphi\in 
\Cc_c(U)$
$$ \ds \int_0^\infty \vphi(x) \nu^\veps(dx) \to \int_0^\infty \vphi(x) \nu(dx) \,. $$

\section{Preliminaries and Assumptions}\label{sec:scaling}

In this section we recall some known results on the BD system together with assumptions for the main 
results of this paper. First of all, we refer the reader to Theorem 2.1 in \cite{Laurencot2002a} for existence and 
uniqueness of (non-negative) global solution to \eqref{sys:BD_rescaled} satisfying the balance of mass  
\eqref{eq:mass}  at fixed $\veps>0$. Well-posedness follows from growth conditions on the kinetic rates, namely we 
assume

\medskip

\begin{hyp}\label{assumption_1}
 The rates $\alpha^\veps$, $\beta^\veps$, $(a_i^\veps)_{i\geq2}$ and $(b_i^\veps)_{i\geq3}$ are positives and, for each 
$\veps>0$, there exists a constant  $K(\veps)>0$ such that 
 \begin{equation*}
  \begin{array}{ll}
   \ds a_{i+1}^\veps - a_i^\veps \leq K(\veps)\,, & i\geq 2\,, \\[0.8em]
   \ds b_i^\veps - b_{i+1}^\veps \leq K(\veps)\,, & i\geq 3\,. 
  \end{array}
 \end{equation*}
\end{hyp}
From now, for each $\veps>0$, we assume $u^\veps$ and $(c_i^\veps)_{i\geq 2}$ are non-negatives and define a solution to
\eqref{sys:BD_rescaled}, that belongs (each) to $\Cc([0,+\infty))$.

\medskip

\noindent We  construct aggregation and fragmentation rates as functions on $\Rb_+$ (similarly to $f^\veps$), namely, for  
each $\veps>0$ we define, for all $x$ in $\Rb_+$,
\begin{equation*} 
  a^\veps(x)  := \ds  \sum_{i \geq 2} a_i^\veps \indic{\Lambda_i^\veps}(x)\,, \text{ and }  b^\veps(x) := \ds  \sum_{i 
\geq 3} b_i^\veps \indic{\Lambda_i^\veps}(x)\,.
\end{equation*}
Now, we are able to derive a weak equation on the density approximation $f^\veps$, for each $\veps>0$, in which we will 
pass to the limit to recover weak solutions to Eq. \eqref{eq:LS}. 

\medskip

\begin{propo}\label{prop:weak_eps}
Under Assumption \ref{assumption_1}, let $\{f^\veps\}$ constructed by 
Eq.~\eqref{eq:def_feps}. For each $\veps>0$, and all $\vphi\in W^{1,\infty}_{loc}(\Rb_+)$ such that 
$\partial_x\vphi \in L^\infty(\Rb_+)$, we have, for all $t\geq 0$,
\begin{multline}\label{eq:weak_form_eps}
 \int_0^{+\infty} f^\varepsilon(t,x) \vphi(x)\, dx  \\   =\int_0^{+\infty} f^{in,\veps}(x) \vphi(x)\, dx +  \int_0^t [ 
 \alpha^\veps u^\veps(s)^2  -  \beta^\veps \veps^{\eta} c_2^\veps(s)] \left(\frac{1}{\veps} \int_{\Lambda_2^\veps}  
 \vphi(x)\, dx\right) ds \\
 + \int_0^t \int_0^{+\infty} \left[ a^\veps(x) u^\veps(s) f^\veps(s,x) \Delta_{\veps}\vphi(x)  -   b^\veps(x) 
 f^\veps(s,x) \Delta_{-\veps}\vphi(x) \right]\, dx\, ds\,,
\end{multline}
where $\Delta_h \vphi(x) = (\vphi(x+h) - \vphi(x))/h$, for $h\in\Rb$, and
\begin{equation} \label{eq:masscons_eps}
 u^\veps(t) + \int_0^\infty x f^\veps(t,x) \, dx = m^\veps.
\end{equation}
\end{propo}
This result follows from \cite[Lemma 4.1]{Laurencot2002a}, which allows taking $\vphi(x)=x$ in the equation. In the next 
assumption we assume standard hypotheses on the convergence of the rate functions and their sub-linear control, see 
also \cite{Collet2002,Laurencot2002a}.

\medskip

\begin{hyp} 
{\normalfont Convergence of the rate functions.}
let $\alpha$ and $\beta$ be two positive numbers, and let $a$ and $b$ be two non-negative continuous functions on 
$[0,+\infty)$ that are positive on $x\in(0,+\infty)$. Then, as $\veps\to 0 $, we suppose that   
\begin{align}
 & \{ \alpha^\veps \} \text{ converges towards } \alpha\,. \tag{H1} \label{H1}\\[0.5em]
 & \{ \beta^\veps \} \text{ converges towards } \beta\,. \label{H2} \tag{H2} \\[0.5em]
 & \{a^\veps(\ldotp)\} \text{ converges uniformly on any compact set of } [0,+\infty) \text{ towards } a(\ldotp) \text{ 
and } \nonumber \\
 & \hspace{3em}\exists K_a>0 \text{ s.t. } a^\veps(x) \leq K_a (1+x), \ \forall x\in \Rb_+ \text{ and } \forall 
\veps>0\,. \label{H3} \tag{H3} \\[0.5em]
 & \{b^\veps(\ldotp)\} \text{ converges uniformly on any compact set of } [0,+\infty) \text{ towards } b(\ldotp) \text{ 
and } \nonumber \\
 & \hspace{3em}\exists K_b>0 \text{ s.t. } b^\veps(x) \leq K_b (1+x), \ \forall x\in \Rb_+ \text{ and } \forall 
\veps>0\,. \label{H4} \tag{H4}
\end{align}
\end{hyp}
We recall a discussion on the scaling of the coefficients is differed to Section \ref{sec:disc}. The next assumption details the behaviour of the rate functions around $0$. This is the essential assumption which 
allow us to identify  the limit of $\veps^\eta c_2^\veps$ in the second integral in the right hand side of 
\eqref{eq:weak_form_eps}.

\medskip

\begin{hyp} \label{hyp:coef_BD} {\normalfont Behavior of the rate functions near $0$.}
 We suppose there exist $r_a \in [0,1)$, $r_b\geq r_a$, $\overline{a}>0$, $\overline{b}>0$ such that
 \begin{equation} \label{H5} \tag{H5}
  \begin{array}{lc|cl}
   a(x) \sim_{0^+} \overline{a}x^{r_a}\,, & & & b(x) \sim_{0^+} \overline{b}x^{r_b}\,, \\[0.8em]
   a^{\veps}(\veps i) = a(\veps i) + o((\veps i)^{r_a})\,,  & & & 
b^{\veps}(\veps i) = b(\veps i) + o((\veps 
i)^{r_b})\,,
  \end{array}
 \end{equation}
 where $o$ is the Landau notation, \textit{i.e.} $o(x)/x\to 0$ as $x \to 0$.
\end{hyp}

\medskip

\noindent Note, if $0 \leq r_b < r_a$ or $r_a\geq 1$, the kinetic rates $a$ and $b$ are related to outgoing 
characteristics for which the theory already exists, see \cite{Laurencot2002a,Collet2002}. Finally, we assume some 
control on the initial conditions. For convenience, we define the quantity
\begin{equation} \label{def:rho}
  \rho := \lim_{x\to 0^+} \frac{b(x)}{a(x)} = \lim_{x\to 0^+} \frac{\overline b}{\overline a} x^{r_b-r_a} \in 
[0,+\infty)\,.
\end{equation}
It determines whether the characteristic at $x=0$ is ongoing or outgoing, according to whether $u(t)$ is greater or less than 
$\rho$ in \eqref{eq:mass}. 

Then, we introduce a set of functions which shall play a key role. We denote by $\Uc$ the set of 
non-negative convex functions $\Phi$ belonging to $\Cc^1([0,+\infty))$ and piecewise  
$\Cc^{2}([0,+\infty))$ such that $\Phi(0)=0$, $\Phi'$ is concave, $\Phi'(0)\geq 0$, and
\[ \lim_{x\to +\infty} \frac{\Phi(x)}{x} = + \infty\,. \]
Note that $\Phi$ is increasing. These functions have remarkable properties when conjugate to the structure of the 
Becker-D\"oring system and provide important estimates, see for instance \cite{Laurenccot2002}. 

\medskip

\begin{hyp} \label{hyp:initial} {\normalfont Initial conditions.}
We assume there exists $u^{\rm in}>\rho$ and a non-negative measure $\mu^{\rm in}\in\Mc_f([0,+\infty))$ such 
that $u^{\rm in,\, \veps}$ converges to $u^{\rm in}$ in $\mathbb R_+$ and $\{f^{\rm in,\, \veps} \}$  converges 
to $\mu^{\rm in}$, in the $weak-*$ topology of $\Mc_f([0,+\infty))$. Moreover, we assume there exists $\Phi\in\Uc$ such 
that 
 \begin{equation} \label{H6}\tag{H6}
   \ds \sup_{\veps>0} \int_0^\infty \Phi(x) f^{\rm in,\, \veps}(x)dx < +\infty \,.
 \end{equation}
 En particular, we can define  
 \[m := u^{\rm in} + \int_0^\infty x \mu^{\rm in}(dx)\,.\]
 Moreover, we suppose that for all $z\in(0,1)$,
 \begin{equation} \label{H7} \tag{H7}
  \sup_{\veps>0} \, \sum_{i\geq 2} \veps^{r_a} c_i^{\rm in,\,\veps} e^{-iz} < +\infty\,. \\
 \end{equation}
\end{hyp}

\medskip

\begin{remk}
$m$ is well-defined since $weak-*$ convergence plus the extra-moment in \eqref{H6} give the limit
\[ \int_0^\infty x f^{\rm in,\, \veps}(dx)\to \int_0^\infty x \mu^{\rm in}(dx)\,. \]
See for instance \cite[Proof of Theorem 2.3]{Collet2002}.
\end{remk}

\medskip

\begin{remk}
In fact, we could obtain freely this $\Phi$ assuming a {\it stronger} weak convergence (against $(1+x)\vphi(x)$ for 
$\vphi$ bounded and continuous). See for instance \cite{Chauhoan1977} for the construction of such a $\Phi$. 
\end{remk}

\medskip

\begin{remk}
We highlight that  condition \eqref{H7} is not restrictive. For example, consider 
$f^{\rm in}(x)=x^{-r}$ on $(0,1)$ and $0$ elsewhere, with $r\leq r_a$. Then, consider $c_i^{\rm in,\, \veps} = 
(i\veps)^{-r}$ for $i\leq 1/\veps$, and $0$ 
elsewhere. We have that $\{f^{\rm in, \, \veps}\}$ trivially converges to $f^{\rm in}$ in the sense of \eqref{H6} and 
it satisfies \eqref{H7}. Note that we do not necessarily require the initial condition is composed of `` very large'' 
clusters (of size $i \gg 1/\veps$). 
\end{remk}

\section{Main results}\label{sec:results}

For the remainder of the paper, we always assume that $\{f^\veps\}$ is constructed by \eqref{eq:def_feps}, that 
$\{u^\veps\}$ is given by the balance \eqref{eq:masscons_eps}, and Assumption \ref{assumption_1} to Assumption 
\ref{hyp:initial} hold true. The next definition extends the notion of a solution to the LS model, Eq. \eqref{eq:LS}, 
with a general boundary condition, or {\it nucleation rate}. 

\medskip

\begin{defn}{\normalfont N-solution.}
Let $T>0$, a function $N \in L^\infty_{loc}(\Rb_+)$ called nucleation rate, $u^{\rm in}>\rho$,  a measure $\mu^{\rm 
in}\in\Mc_f([0,+\infty))$, and a measure-valued function $\mu\in L^\infty([0,T];\Mc_f([0,+\infty))$. We say that $\mu$ 
is a $N$-solution of the LS equation (in measure) on $[0,T]$ with mass $m$, when:

\smallskip

\noindent i) There exists a non-negative $u\in\Cc([0,T])$, such that $u(0)=u^{\rm in}$, 
\begin{equation*}
\inf_{t\in[0,T]} u(t)>\rho\,, \text{ and } \ \forall t\in[0,T]\,,\  u(t) + \int_0^\infty x \mu_t(dx) = m\,.
\end{equation*}
%

\smallskip

\noindent ii) For all $\varphi \in \Cc_c^1([0,T)\times[0,+\infty))$ and $t\in[0,T]$
\begin{multline}\label{eq:weak_LS_fast2}
   \int_{0}^T \int_{0}^{\infty} \big[ \partial_t \vphi(t,x) + (a(x)u(t)-b(x))\partial_x \vphi(t,x)\big]\mu(t,dx)\, dt \\
   + \int_0^\infty \vphi(0,x)\mu^{\rm in}(dx) + \int_0^T \vphi(s,0) N(u(s)) \,ds= 0 \,,
 \end{multline}
%
%
\end{defn}
We now state our main results. The first theorem, when $\eta>r_a$, corresponds to the case where the 
first fragmentation rate is too slow and does not contribute to the boundary value. Thus the nucleation rate is proportional to the 
number of encounter of free particles, namely $u(t)^2$ at time $t$. 

\medskip

\begin{thm} \label{thm:LS_slow}{\normalfont The {\it slow} de-nucleation case.}
Assume $\eta > r_a$ and let a sequence $\{\veps_n\}$ converging to $0$. There exists $T>0$,  a sub-sequence 
$\{\veps_{n'}\}$ of $\{\veps_n\}$, and $\mu$ a $N$-solution of LS 
with mass $m$, such that
\[f^{\veps_{n'}} \xrightharpoonup[n'\to+\infty]{} \mu\]
in $\Cc([0,T];w-*-\Mc_f([0,+\infty))$, and, for all $u\geq 0$,
\[ N(u) = \alpha u^2\,. \]
\end{thm}
 
\begin{remk}
The space $\Cc([0,T];w-*-\Mc_f([0,+\infty)))$ has to be understand as measure-valued 
function that are continuous in time for the $weak-*$ topology on $\Mc_f([0,+\infty))$, {\it i.e.} for $\{\nu_t\} \in 
\Cc([0,T];w-*-\Mc_f([0,+\infty)))$, we have, for all $t\in[0,T]$ and $\vphi\in\Cc_c([0,+\infty))$,
\[t\mapsto \int_0^\infty \vphi(x)\nu_t(dx)\]
is continuous.
\end{remk}

\medskip 

\noindent The second theorem holds in the limit case when $\eta=r_a$, \ie~ the first fragmentation rate has the same order of 
magnitude than the aggregation rate ($i\geq2$). Compared to the first case, the nucleation rate is balanced by a function varying between $0$ 
and $1$.

\medskip

\begin{thm} \label{thm:LS_compensated} {\normalfont The {\it compensated} de-nucleation case.}
Assume $\eta =r_a$ and let a sequence $\{\veps_n\}$ converging to $0$. There exists $T>0$,  a sub-sequence 
$\{\veps_{n'}\}$ of $\{\veps_n\}$, and $\mu$ a $N$-solution of LS 
with mass $m$, such that
\[f^{\veps_{n'}} \xrightharpoonup[n'\to+\infty]{} \mu\]
in $\Cc([0,T];w-*-\Mc_f([0,+\infty))$, and, for all $u\geq 0$,
%
 \begin{equation*} 
  N(u) = \begin{cases}
	  \ds \alpha u^2 \frac{  u}{ u+\beta/(\bar a 2^{\eta})}\,, & if \ \eta=r_a <r_b\,,  \\[0.8em]
          \ds \alpha u^2 \frac{\overline{a}u-\overline{b}}{\overline{a}u-\overline{b}+\beta/2^\eta}\,, & if \  
\eta=r_a=r_b\,, 
         \end{cases}
 \end{equation*}
\end{thm}

\begin{remk}
 In the pure aggregation case, with $\beta^\veps=b_i^\veps=0$, then $b=0$ and $\beta=\overline b=0$. Our results in 
Theorem \ref{thm:LS_slow} and Theorem \ref{thm:LS_compensated} are consistent and remain true. 
\end{remk}

\medskip

\noindent Finally, the last theorem considers the case of a fast de-nucleation rate so that the flux at 
the boundary vanished, and the solution can reveal fast oscillation near $x=0$.

\medskip

\begin{thm} \label{thm:LS_fast} {\normalfont The {\it fast} de-nucleation rate.}
Assume $\eta <r_a$ and let a sequence $\{\veps_n\}$ converging to $0$. There exists $T>0$,  a sub-sequence 
$\{\veps_{n'}\}$ of $\{\veps_n\}$, and $\mu$ a $N$-solution of LS 
with mass $m$, such that
\[f^{\veps_{n'}} \xrightharpoonup[n'\to+\infty]{} \mu\]
in $weak-*-L^\infty(0,T;\Mc_f([0,+\infty))$, and, for all $u\geq 0$,
%
\[N(u)=0\,.\]
\end{thm}

\begin{remk}
In this case were not able to prove equicontinuity  of the density approximation in $\Mc_f([0,+\infty))$. For this 
case, in fact, it is true for  $\Mc_f((0,+\infty))$ (open in $x=0$). Also, we use the $weak-*$ topology on 
$L^\infty(0,T;\Mc_f([0,+\infty))$ which is the topology of the point-wise convergence against test functions in 
$L^1(0,T;\Cc_c([0,+\infty))$. 
\end{remk}

\medskip

\begin{remk}
 These limit theorems provide local in time existence and could be extended to a maximal time interval $[0,T)$ where 
$T =\sup \{ \tau \, :\, \inf_{t\in[0,\tau]} u(t) > \rho \}$. Also, uniqueness is not investigate here, but 
and appropriate result would entails convergence of the whole sequence without extraction. 
\end{remk}

\section{The compactness estimates}\label{sec:compact}

In this section we provide the main estimates to obtain sufficient compactness arguments to pass to the limit in 
\eqref{eq:weak_form_eps}-\eqref{eq:masscons_eps}. Remark for further estimations, under \eqref{H1} and \eqref{H2}, 
there exists a positive $K_{\alpha,\beta}$ such that, for all $\veps>0$,
\begin{equation}\label{bound_a1_b2}
  \alpha^{\veps}, \beta^{\veps}, \alpha, \beta \in (0,K_{\alpha,\beta}]\,,
\end{equation}
and \eqref{H3}-\eqref{H4} imply the limit functions also satisfy 
\begin{align}\label{linear_bound}
a(x) \leq K_a (1+x)\text{ and } b(x) \leq K_b (1+x), \ \forall x\in [0,+\infty).
\end{align}
We fix these constants for the remainder. 

\subsection{Uniform bound for the density approximation}

The first lemma gives basic estimates. In particular, it constructs the compact set of $\Mc_f([0,+\infty))$ in which the sequence of
solutions remains.

\medskip

\begin{lem}\label{lem:u1xphi1}
For all $T>0$,
\begin{align}
  & \sup_{\veps>0} \, \sup_{t\in[0,T]} \,  \int_0^{+\infty} (1+x+\Phi(x))f^\veps(t,x)\, dx < +\infty\,, 
\label{eq:phi_1} \\  
  & \sup_{\veps>0} \, \sup_{t\in[0,T]} \, u^\veps(t) < +\infty\,, \vphantom{\int_0^t } \label{bound_u} \\ 
  & \sup_{\veps>0} \, \int_0^T \veps^\eta c_2^\veps(t) \, dt < +\infty\,. \label{bound_c_2_measure}
\end{align}
\end{lem}

\begin{remk}
Similar estimates can be found in \cite{Laurencot2002a} for a different scaling. For sake of completeness we recall the proof below. Note that estimate \eqref{bound_c_2_measure}, although trivial, seems to have not been reported elsewhere, and will be important for the next.
\end{remk}

\medskip

\begin{proof}
 By Assumption \ref{hyp:initial}, the convergence of $\{f^{\rm in,\,\veps}\}$ implies that the sequence lies in 
a $weak-*$ compact set of $\Mc_f([0+\infty))$, and with \eqref{H6} we have
 \begin{equation} 
  \sup_{\veps>0}  \int_{\Rb_+}f^{in,\veps}(x)(1+x+\Phi(x)) dx  < +\infty\,. \label{eq:cond_phi1}
 \end{equation}
 Let us start now with estimate \eqref{bound_u}. By the mass conservation relationship \eqref{eq:masscons_eps}, 
$u^\veps(t) \leq m^\veps$, for any $t\geq0$, and thanks to Assumption \ref{hyp:initial}, $(m^\veps)$ converges as $\veps\to 0$, thus it is 
bounded by a constant $K_m>0$. Then estimate \eqref{bound_u} directly follows. Similarly, we obtain
 \[\sup_{\veps>0} \sup_{t\in[0,T]} \int_0^{+\infty} xf^\veps(t,x)\, dx < +\infty\,.\]
Then, taking $\vphi = \indic{}$ in Eq. \eqref{eq:weak_form_eps}, it immediately yields by re-arranging the non-positive 
term
 \begin{equation*} 
0\leq \int_0^{+\infty} f^\varepsilon(t,x) \, dx + \int_0^t \beta^\veps \veps^\eta c_2^\veps(s) \, ds \leq  
\int_0^{+\infty} f^{in,\veps}(x)\, dx + \int_0^t \alpha^\veps u^\veps(s)^2 \, ds.
\end{equation*}
Using the bounds  \eqref{bound_a1_b2}, \eqref{bound_u} and \eqref{eq:cond_phi1}, we obtain the inequality 
\eqref{bound_c_2_measure} together with the first part of estimate \eqref{eq:phi_1}. 

\noindent Finally, we put $\vphi=\Phi$ in \eqref{eq:weak_form_eps}. Remark that the derivative $\Phi'$ is not uniformly 
bounded, thus we cannot use Proposition \ref{prop:weak_eps} straightforwardly. However, with a classical regularizing 
argument, one can show that the next computations hold true {\it a posteriori}, see for instance \cite[proof of Lemma 
4.2]{Laurencot2002a}. We remark that 
\[ 0\leq \Delta_{\veps}\Phi(x) \leq \Phi'(x+\veps), \quad   -\Delta_{-\veps}\Phi(x) \leq -\Phi'(x) \leq 0. \]
Moreover, $\Phi'(x+\veps) \leq \Phi'(x) + \veps \Phi''(0)$. Thus, dropping the non-positive term, using 
\eqref{H3} and again that $u^\veps(t) \leq K_m$,
\begin{multline} \label{eq:weak_form_phi1}
\int_0^{+\infty} f^\varepsilon(t,x) \Phi(x)\, dx \leq  \int_0^{+\infty} f^{in,\veps}(x) \Phi(x)\, dx + \int_0^t  
\alpha^\veps u^\veps(s)^2 \left(\frac{1}{\veps}\int_{\Lambda_2^\veps} \Phi(x)dx\right) \, ds \\
 + K_m K_a\int_0^t \int_0^{+\infty}  (1+x) f^\veps(s,x) (\Phi'(x) + \veps \Phi_{1,r}''(0)) \, dx\, ds,
\end{multline}
Let $\delta>0$. Note that $x\Phi'(x)\leq2\Phi(x)$ (see \cite[Lemma A.1]{Laurencot2001}), we get
\begin{multline*}
 \int_0^{+\infty}  (1+x)  f^\veps(s,x) \Phi'(x) \, dx  \leq   \int_0^{\delta}  f^\veps(s,x) \Phi'(x) \, dx + \left( 
\frac 1 \delta +1\right) \int_0^{+\infty}   x f^\veps(s,x) \Phi'(x) \, dx  \\
 \leq \big(\sup_{(0,\delta)} \Phi' \big)\int_0^{\infty}  f^\veps(s,x) \, dx + 2\left( \frac 1 \delta +1\right) 
\int_0^{+\infty}    f^\veps(s,x) \Phi(x) \, dx.
\end{multline*}
We introduce this last estimation into Eq.~\eqref{eq:weak_form_phi1} and we conclude using the previous bounds and 
Gr\"onwall lemma. 

\end{proof}

\subsection{Pointwise estimations on the density}

We turn now to the main estimate of this paper. Indeed, to obtain equicontinuity for the density $\{f^\veps\}$ (in a 
measure space), and then identify the boundary condition, we need to control the behaviour of the small-sized clusters, 
particularly because of the term $\veps^{\eta } c_2^{\veps}$ in the weak equation \eqref{eq:weak_form_eps}. Remark that 
we already have a weak bound (in time) given by Eq.~\eqref{bound_c_2_measure}. In the next Proposition 
\ref{prop:bound_laplace} we improve this estimate by a control on exponential moments which depends on $\rho$ (defined 
in Eq. \eqref{def:rho}). Moments are classical tools and 
play a key role in the well-posedness of BD theory. More recently, exponential moments were also used 
\cite{Jabin2003,Canizo2013} to study long time behavior of BD solutions. Here, let us define the 
discrete Laplace transform  
\begin{equation} \label{def_laplace}
 F^\veps(t,z) = \sum_{j\geq 2} \veps^{r_a} c_j^\veps(t) e^{-jz}\,, \ z\in(0,1)\,.
\end{equation}
From the re-scaled system \eqref{sys:BD_rescaled}, the sequence $(d_i^\veps)_{i \geq 2}$ defined by $d_i^\veps 
:=\veps^{r_a} c_i^\veps$, for $i\geq 2$, satisfies, for each $\veps>0$, the following equations 
\begin{equation} \label{eq:BD_system_H}
  \veps^{1-r_a}\frac{d}{dt}d_i^\veps(t) = H_{i-1}^\veps - H_i^\veps\,, \quad i\geq 2\,,
\end{equation}
where the fluxes are
\begin{equation*}
  \ds H_1^\veps =\alpha^\veps u^\veps(t)^2 - \beta^\veps \veps^{\eta-r_a} d_{2}^\veps(t)\,, \text{ and }\ds H_i^\veps = 
 \overline a_i^\veps u^\veps(t) d_i^\veps(t) - \veps^{r_b-r_a} \overline b_{i+1}^\veps 
d_{i+1}^\veps(t)\,, \quad i\geq 2\,,
\end{equation*}
with, for all $i\geq2$,
\begin{equation*}
\overline a_i^\veps = \frac{a_i^\veps}{\veps^{r_a}}\,, \quad \text{and }\quad\overline b_{i+1}^\veps = 
\frac{b_{i+1}^\veps}{\veps^{r_b}}\,.
\end{equation*}
Note that, under hypotheses \eqref{H3}, \eqref{H4} and \eqref{H5}, the kinetic coefficients $\alpha^\veps$, 
$\beta^\veps$ and $\overline a_i^\veps$, $\overline b_i^\veps$, $i\geq 2$, are convergent sequences toward a positive 
value (resp. $\alpha$, $\beta$, $\overline a i^{r_a}$, $\overline{b} i^{r_b}$).

\medskip

\begin{propo}\label{prop:bound_laplace}
Let $T>0$ and $\{\veps_n\}$ a sequence converging to $0$ such that $\{u^{\veps_n}\}$ converges toward $u$ 
uniformly on $[0,T]$, with  $\inf_{t\in [0,T]} u(t) > \rho$. There exists $z_0>0$ such that for all 
$z\in(0,z_0)$
 \begin{equation}\label{bound_ci_caseII}
  \sup_{n\geq0} \, \sup_{t\in[0,T]}F^{\veps_{n}}(t,z)<\infty\,.
 \end{equation}
 In particular, for all $r \geq r_a$ and  $i\geq 2$, we have
 \begin{equation}\label{eq:bound_ci_r}
  \sup_{n\geq 0} \, \sup_{t\in[0,T]} \, \veps^r c_i^{\veps_{n}}(t) < +\infty \,.
 \end{equation}
\end{propo}

\medskip

\begin{remk}
It is immediate from estimate \eqref{eq:bound_ci_r} that we can obtain compactness in $w-*-L^\infty(0,T)$ for any 
finite size cluster $\veps^{r}c_i^\veps$, which will be used to prove theorem \ref{thm:LS_slow} 
and \ref{thm:LS_compensated}.
\end{remk}

\medskip

\begin{remk}
We cannot prove that the pseudo-moment $F^\veps$ is propagated along limit solution for which $u(t) \leq \rho$ on some 
time interval. This is important in the case $r_a=r_b$ since $\rho>0$ and $u$ can eventually cross this threshold 
(which is, up to our knowledge, an open problem).
\end{remk}

\medskip

\begin{proof}
Let $z>0$ and $\veps>0$. First, note the discrete Laplace transform define in Eq. \eqref{def_laplace} is finite for each 
$\veps>0$ and for all $t$ in 
$[0,T]$, since 
\[ F^\veps(t,x) \leq \veps^{r_a-1} \int_0^\infty x f^\veps(t,x)\, dx\,.\]
Let us derive $F^\veps$ with respect to time (derivation under the sum is justified by similar bound). For all 
$t\in[0,T]$, we get
\begin{equation*}
 \veps^{1-r_a} \partial_t  F^\veps(t,z) = \sum_{j\geq 2} e^{-jz} \big[ H_{j-1}^\veps - H_j^\veps\big] = e^{-2z} 
H_1^\veps - (1-e^{-z}) \sum_{j\geq 2} e^{-jz} H_j^\veps\,.
\end{equation*}
Thus, developing the fluxes we get
\begin{multline*}
 \veps^{1-r_a} \partial_t  F^\veps(t,z) =  e^{-2z} H_1^\veps - (1-e^{-z}) \sum_{j\geq 2} e^{-jz} \overline a_j^\veps 
u^\veps(t) d_j^\veps(t)\\ 
 + (1-e^{-z}) \sum_{j\geq 2} e^{-jz} \veps^{r_b-r_a} \overline b_{j+1}^\veps d_{j+1}^\veps(t) \,.
\end{multline*}
Then, re-indexing the second sum on the right hand side, we obtain
\begin{multline} \label{dt_Feps_intermediate}
 \veps^{1-r_a} \partial_t  F^\veps(t,z) =  e^{-2z} H_1^\veps - (1-e^{-z}) e^{-2z} \overline a_2^\veps u^\veps(t) 
d_2^\veps(t) \\
 - (1-e^{-z}) \sum_{j\geq 3} e^{-jz} \overline a_j^\veps \left[ u^\veps(t) - \frac{b_{j}^\veps}{a_j^\veps } e^z \right] 
d_j^\veps(t)   \,.
\end{multline}
Since $\inf_{t\in [0,T]} u(t) > \rho$, we can find a constant $c$ such that $\inf_{t\in [0,T]} u(t)\geq c>\rho$. Then, 
by uniform convergence of $\{u^{\veps_n}\}$, there exists $\tilde \veps>0$ small enough, such that  for all $n$ with 
$\veps_n\leq \tilde \veps$, $\inf_{t\in[0,T]}\, u^{\veps_n}(t) \geq c>\rho$. Also, we can choose $\delta>0$ and $z_0>0$, both small 
enough, such that for all $t\in[0,T]$ we have $c > \rho e^{z_0} + 2\delta$. Then, there exists $N>0$ such that, for all $z\in(0,z_0)$
\[ \inf_{n\geq N}\, \inf_{t\in[0,T]}\, u^{\veps_n}(t) > \rho e^{z} + 2\delta \,.\]
Then, by hypothesis \eqref{H5}, for all $3\leq j \leq 1/\sqrt{\veps}$,
\[\frac{b_{j}^\veps}{a_j^\veps } = \frac{\overline b}{\overline a} \frac{(\veps j)^{r_b}+o((\veps j)^{r_b})}{(\veps 
j)^{r_a}+o((\veps j)^{r_a})} = \frac{\overline b}{\overline a} (\veps j)^{r_b-r_a} (1+o(1)) \,, \]
so that, we have, for $N$ large enough,  
\[\sup_{n\geq N}\, \sup_{j\in[3,\ldots,\lfloor1/\sqrt{\veps_n}\rfloor-1]} \left\lvert 
\rho-\frac{b_{j}^{\veps_n}}{a_j^{\veps_n} 
}\right\rvert < \delta  e^{-z}\,. \]
The latter gives a uniform control in $j$ for the relatively ``small'' sizes $j\leq 1/\sqrt{\veps}$. We separate the 
sum in Eq.~ \eqref{dt_Feps_intermediate} in two parts, the small-sized clusters for $j\in 
(3,\ldots,\lfloor1/\sqrt{\veps_n}\rfloor-1)$ in one side, for which (for $n\geq N$)
\[ u^{\veps_n}(t) -  \frac{b_{j}^{\veps_n}}{a_j^{\veps_n}} e^z  = u^{\veps_n}(t) - \rho e^{z} + e^z 
\left(\rho - \frac{b_{j}^{\veps_n}}{a_j^{\veps_n}} \right)  \geq 2\delta - \delta = \delta\,,\] 
and the large-sized clusters in another side. Hence, for all $t\in[0,T]$,
\begin{multline} \label{dt_Feps_intermediate_1}
 \sum_{j\geq 3} e^{-jz}  \overline a_j^{\veps_n} \left[  u^{\veps_n}(t) -  
\frac{b_{j}^{\veps_n}}{a_j^{\veps_n} } e^z \right] 
d_j^{\veps_n}(t) \\
 \geq \delta \sum_{j=3}^{\lfloor1/\sqrt{\veps_n}\rfloor-1} e^{-jz}  \overline a_j^{\veps_n} d_j^{\veps_n}(t)  
+ \sum_{j\geq 
\lfloor 1/\sqrt{\veps_n}\rfloor} e^{-jz} \overline a_j^{\veps_n} \left[  u^{\veps_n}(t) -  
\frac{b_{j}^{\veps_n}}{a_j^{\veps_n} } e^z 
\right] d_j^{\veps_n}(t)\,.
\end{multline}
Using hypothesis \eqref{H5}, there exists $x_0$ such that for all $x\in(0,x_0)$, $a(x)/x^{r_a}>3\overline a/4$. Thus, 
there exists $\tilde N$ such that for all $n\geq \tilde N$ and for all $2\leq i\leq 1/\sqrt{\veps_n}$ we have $\veps_n i 
\leq \sqrt{\veps_n}< x_0$ and $a(\veps i)/(\veps_ni)^{r_a}\geq 3\overline a/4$. Still with hypothesis \eqref{H5}, we can choose $\tilde N$ such that for all $n>\tilde N$,  and for all $2\leq i\leq 1/\sqrt{\veps_{n}}$, we have $a^\veps(\veps i)/(\veps_ni)^{r_a}\geq \overline a/2$. Hence, from the rank $\tilde N$, there exists a 
constant $\widetilde K_a>0$ such that for all $n\geq \tilde N$ and for all $2\leq j \leq 1/\sqrt{\veps_n}$, we have
\[\overline a_j^{\veps_n} = \frac{a_j^{\veps_n}}{\veps_n^{r_a}} \geq  \widetilde K_a := \frac 1 2 \overline 
a\, 2^{r_a}\,.\]
Accordingly, the rest of the proof has to be understood for $n$ large enough. Using the equation on $H_1^\veps$ and plugging inequality \eqref{dt_Feps_intermediate_1} into 
Eq.~\eqref{dt_Feps_intermediate} we obtain 
\begin{multline*} 
 \veps_n^{1-r_a} \partial_t  F^{\veps_n}(t,z) \leq   e^{-2z} [\alpha^{\veps_n} u^{\veps_n}(t)^2 
-\veps_n^{\eta-r_a} \beta^{\veps_n }
d_{2}^{\veps_n}(t)]
 \\ - (1-e^{-z}) e^{-2z} [ \overline a_2^{\veps_n} u^{\veps_n}(t) - \delta \tilde K_a]  d_2^{\veps_n}(t) - 
(1-e^{-z}) \delta 
\widetilde K_a  \sum_{j= 2}^{\lfloor1/\sqrt{\veps_n}\rfloor-1} e^{-jz} d_j^{\veps_n}(t)  \\
  - (1-e^{-z}) \sum_{j\geq \lfloor1/\sqrt{\veps_n}\rfloor} e^{-jz} \overline a_j^{\veps_n} \left[  
u^{\veps_n}(t) -  
\frac{b_{j}^{\veps_n}}{a_j^{\veps_n}} e^z \right] d_j^{\veps_n}(t)\,.\\
\end{multline*}
Remark that $ \overline a_2^{\veps_n} u^{\veps_n}(t) - \delta \widetilde K_a \geq \tilde K_a( \rho e^z + 2 \delta 
-\delta) \geq  \widetilde K_a \rho \geq 0$. Using the moment estimates \eqref{bound_u} and hypothesis \eqref{H3}, we 
have $\sup_{t\in[0,T]} \alpha^\veps u^\veps(t)^2 \leq K_0$ uniformly in $\veps>0$. Thus, dropping also some negative 
terms, we have
\begin{multline*} 
 {\veps_n}^{1-r_a} \partial_t  F^{\veps_n}(t,z)  \leq  K_0 e^{-2z} - (1-e^{-z}) \delta \widetilde K_a \sum_{j= 
2}^{\lfloor1/\sqrt{{\veps_n}}\rfloor-1} e^{-jz} d_j^{\veps_n}(t)   \\
 + (1-e^{-z}) \sum_{j\geq \lfloor1/\sqrt{{\veps_n}}\rfloor} e^{-jz}\frac{b_{j}^{\veps_n}}{{\veps_n}^{r_a}} 
d_j^{\veps_n}(t)  \,.
\end{multline*}
Now using that 
\[\sum_{j= 2}^{ \lfloor1/\sqrt{{\veps_n}}\rfloor-1} e^{-jz} d_j^{\veps_n}(t) =F^{\veps_n}(t,z)-\sum_{j \geq  
\lfloor1/\sqrt{{\veps_n}}\rfloor } e^{-jz} d_j^{\veps_n}(t)\,,\] 
we obtain 
\begin{multline*} 
  {\veps_n}^{1-r_a} \partial_t  F^{\veps_n}(t,z) \leq   K_0 e^{-2z} - (1-e^{-z}) \delta \widetilde K_a 
F^{\veps_n}(t,z) \\
  + (1-e^{-z}) \delta  \sum_{j \geq \lfloor1/\sqrt{{\veps_n}}\rfloor } e^{-jz}  \widetilde K_a d_j^{\veps_n}(t)  
  + (1-e^{-z})e^z \sum_{j\geq \lfloor1/\sqrt{{\veps_n}}\rfloor} e^{-jz}\frac{b_{j}^{\veps_n}}{{\veps_n}^{r_a}} 
d_j^{\veps_n}(t)  \,.
\end{multline*}
At this step, we recall that by definition we have, for all $j\geq 2$, $d_j^\veps/\veps^{r_a} = c_j^\veps$, and 
$\widetilde K_a< a_j^\veps/\veps^{r_a}$, so that, with $K = \max(\delta,e^z)$, 
\begin{multline*} 
 {\veps_n}^{1-r_a} \partial_t  F^{\veps_n}(t,z)  \leq  K_0 e^{-2z} - (1-e^{-z})  \delta \widetilde K_a 
F^{\veps_n}(t,z)  \\
 + (1-e^{-z}) K  \sum_{j\geq \lfloor1/\sqrt{{\veps_n}}\rfloor} e^{-jz}  (a_j^{\veps_n} + b_{j}^{\veps_n}) 
c_j^{\veps_n}(t) \,.
\end{multline*}
Finally, by hypotheses \eqref{H3}-\eqref{H4}, we have, for all $j\geq \lfloor1/\sqrt{\veps}\rfloor$ (and $\veps$ small enough)
\[ e^{-jz}  (a_j^\veps + b_{j}^\veps) \leq (K_a+K_b)(1+\veps j)  e^{-jz} \leq (K_a+K_b) \veps\,.\]
Thus,
\begin{multline*} 
 {\veps_n}^{1-r_a} \partial_t  F^{\veps_n}(t,z) \leq   K_0 e^{-2z} - (1-e^{-z}) \delta \widetilde K_a 
F^{\veps_n}(t,z) \\
+ (1-e^{-z}) K (K_a+K_b)   \int_0^{+\infty} f^{\veps_n}(t,x)dx \,.
\end{multline*}
By the moment estimates \eqref{eq:phi_1}, there exists $\tilde K$ independent from $\veps_n$ such that
\begin{equation}\label{ineq:F}
 {\veps_n}^{1-r_a} \partial_t  F^{\veps_n}(t,z)  \leq  - (1-e^{-z}) \delta \widetilde K_a F^{\veps_n}(t,z) + 
\tilde K \,.
\end{equation}
%
We can conclude that
\begin{equation*}
 F^{\veps_n}(t,z) \leq F^{\veps_n}(0,z) + \frac{ \widetilde K}{\delta \widetilde K_a (1-e^{-z})}\,,
\end{equation*}
and the result \eqref{bound_ci_caseII} follows thanks to the initial bound on $ F^\veps(0,z)$ given by hypothesis 
\eqref{H7}. Note that \eqref{eq:bound_ci_r} directly follows from the previous bound \eqref{bound_ci_caseII} and the 
definition of the discrete Laplace transform \eqref{def_laplace}.
\end{proof}

\medskip

\begin{remk}
Estimate \eqref{ineq:F} on $F^\veps$ can be easily generalized for any exponent $r$ instead of $r_a$. 
Writing 
$G^{^\veps}(t,z)= \sum_{j\geq 2} \veps^{r} c_j^\veps(t) e^{-jz}$, and following the same steps, we find
 \begin{equation*}
   \veps^{1-r_a} \partial_t  G^\veps(t,z)  \leq  - (1-e^{-z}) \delta \widetilde K_a G^\veps(t,z) + \veps^{r-r_a}\tilde 
K \,.
\end{equation*}
Thus, this inequality provides valuable information if $r\geq r_a$.
\end{remk}

\subsection{Equicontinuity lemmas}

We now turn to the equicontinuity of the density approximation, as a measure valued time-dependent function. The new 
result here is to provide equicontinuity in a measure space on $[0,\infty)$ (lemma \ref{lem:equicontinuity}) . The 
first lemma is independent on $\eta$ and  similar to \cite{Laurencot2002a,Collet2002}.
\medskip

\begin{lem}\label{equicontinuity_u}
Let $T>0$. The family $\{u^{\veps}\}$ is equicontinuous on $[0,T]$.
\end{lem}

\medskip

\begin{proof}
Let us fix $T>0$. From the mass conservation \eqref{eq:masscons_eps}, we can deduce that the equicontinuity of 
$\{u^{\veps}\}$ directly follows from the one of the sequence $\{\int_0^{+\infty}xf^{\veps} (\cdot,x)\,dx\}$. Thus, we 
focus on this latter. We have, from Eq.~\eqref{eq:weak_form_eps}  with $\vphi(x)=x$, for all  $t\in [0,T-h]$ and $s \in 
[0,h]$ with $0 < h < T$,
\begin{multline} \label{eq:equicontinuity_u1}
 \left| \int_0^{+\infty} [f^\varepsilon(t+s,x)-f^\varepsilon(t,x) ] x\, dx \right|\leq \left(\frac 1 \veps 
\int_{\Lambda_2^\veps} x\, dx \right) \int_t^{t+s}  ( \alpha^\veps u^\veps(\sigma)^2  + \beta^\veps\veps^{\eta} 
c_2^\veps(\sigma))  \,d\sigma \\
 + \int_t^{t+s} \int_0^{+\infty} \vert a^\veps(x) u^\veps(\sigma) f^\veps(\sigma,x)   -  b^\veps(x) 
f^\veps(\sigma,x)\vert\, dx\, d\sigma\,.
\end{multline}
The first term in the r.h.s of \eqref{eq:equicontinuity_u1} can be bounded, thanks to the bound \eqref{bound_a1_b2}, by 
\begin{multline*}
\left(\frac 1 \veps \int_{\Lambda_2^\veps} x\, dx \right)\int_t^{t+s} ( \alpha^\veps u^\veps(\sigma)^2  + 
\beta^\veps\veps^{\eta} c_2^\veps(\sigma)) \,d\sigma \\
\leq 2K_{\alpha,\beta} \left[\veps\sup_{t\in[0,T]} u^\veps(t)^2 + \sup_{t\in[0,T]} \veps^{\eta+1} c_2^{\veps}(t) 
\right] h\,.
\end{multline*}
Then, since $\eta\geq 0$ and remarking that $\veps c_2^\veps$ is obviously bounded by the $L^1$ norm of $f^\veps$, we 
can use the moment estimates in Eqs.~\eqref{eq:phi_1} and \eqref{bound_u}, so that for $\veps$ sufficiently small, 
there exists $K$ independent of $t$ and $\veps$ such that
\begin{equation}\label{equicontinuity_u_first} 
\left(\frac 1 \veps \int_{\Lambda_2^\veps} x\, dx \right)\int_t^{t+s}  ( \alpha^\veps u^\veps(\sigma)^2  + \beta^\veps 
\veps^{\eta} c_2^\veps(\sigma))  \,d\sigma \leq K h\,.
\end{equation}
Let us now focus on the second term on the right-hand side of Eq.~\eqref{eq:equicontinuity_u1}. Using hypotheses 
\eqref{H3}-\eqref{H4} and the moment estimates in Eq.~\eqref{eq:phi_1}, we get
\begin{multline*} \int_t^{t+s} \int_0^{+\infty} \vert a^\veps(x) u^\veps(\sigma) f^\veps(\sigma,x)   -  b^\veps(x) 
f^\veps(\sigma,x)\vert\, dx\, d\sigma\\
\leq \left(K_a\sup_{\veps>0}\sup_{t\in[0,T]} u^\veps(t) + K_b \right) \int_t^{t+s}  \int_0^{+\infty} f^\veps(\sigma,x) 
(1+x)\, dxd\sigma\,.
\end{multline*}
Hence, there is a constant $K>0$ such that 
\begin{multline} \label{equicontinuity_u_second} \int_t^{t+s} \int_0^{+\infty} \vert a^\veps(x) u^\veps(\sigma) 
f^\veps(\sigma,x)   -  b^\veps(x) f^\veps(\sigma,x)\vert\, dx\, d\sigma\\
\leq  h K \left( \sup_{\veps>0} \, \sup_{t\in[0,T]} \int_0^{+\infty} (1+x)f^\veps(t,x)\, dx\right)\,.
\end{multline}
Combining both inequalities \eqref{equicontinuity_u_first}-\eqref{equicontinuity_u_second}, it follows that for all 
$\delta>0$, for all  $h\in(0,T)$ ,
\begin{equation*}
\sup_{\veps>0} \sup_{t \in[0,T-h]} \sup_{s \in[0,h]} \left| \int_0^{+\infty} [f^\varepsilon(t+s,x)-f^\varepsilon(t,x) ] 
x\, dx \right|\leq \delta\,,
\end{equation*}
which gives the equicontinuity property for $\{u^\veps\}$.
\end{proof}

\medskip

The next lemma improves the equicontinuity of $\{f^\veps\}$ around $x=0$.

\medskip

\begin{lem}\label{lem:equicontinuity}
Assume $\eta \geq r_a$ and $T>0$. Let $\{\veps_n\}$ a sequence converging to $0$ such that $\{u^{\veps_n}\}$ 
converges toward $u$ uniformly on $[0,T]$ satisfying  $\inf_{t\in [0,T]} u(t) > \rho$. Then the sequence $\{f^{\veps_{n}}\}$ 
is equicontinuous in $\Mc_f([0,+\infty))$.
\end{lem}

\medskip

\begin{proof} Let us fix $T>0$. Let $h \geq 0 \in(0,T)$, $t\in[0,T-h]$ and
$s\in[0,h]$ we have, for all $\psi\in\Cc^\infty_c([0,+\infty))$ and $\veps>0$
\begin{multline} \label{eq:equicontinuity_1}
\left| \int_0^{+\infty} [f^\varepsilon(t+s,x)-f^\varepsilon(t,x)] \psi(x)\, dx \right|  \\
\leq   \int_t^{t+s}  (\alpha^\veps u^\veps(\sigma)^2  + \beta^\veps \veps^{\eta} c_2^\veps(\sigma)) 
\left(\frac{1}{\veps}\int_{\Lambda_2^\veps}  \left\vert  \psi(x) \right\vert \, dx\right) \,d\sigma \\
 + \int_t^{t+s} \int_0^{+\infty} \vert a^\veps(x) u^\veps(\sigma) f^\veps(\sigma,x) \Delta_{\veps}\psi(x)  -  
b^\veps(x) f^\veps(\sigma,x) \Delta_{-\veps}\psi(x) \vert\, dx\, d\sigma\,.
\end{multline}
The first integral in the right-hand side can be bounded as follows
\begin{multline*}
\int_t^{t+s}  ( \alpha^\veps u^\veps(\sigma)^2  + \beta^\veps \veps^{\eta}c_2^\veps(\sigma))   \left(\frac{1}{\veps} 
\int_{\Lambda_2^\veps} \left\vert  \psi(x) \right\vert\, dx\right) \,d\sigma  \\
\leq h  \| \psi \|_{\infty} \sup_{t\in[0,T]} \left[ \alpha^\veps u^\veps(t)^2 +   \beta^\veps \veps^{\eta} 
c_2^{\veps}(t) \right]\,.
\end{multline*}
Using Eqs.~\eqref{bound_a1_b2}, \eqref{bound_u} and by Proposition \ref{prop:bound_laplace}, Eq.~\eqref{eq:bound_ci_r}, 
both terms in the supremum are uniformly bounded in time and along $\{\veps_{n}\}$. Hence, there exists $K$ independent of $T$ and 
$\veps$ such that, for all $t\leq T-h$, $s\in[0,h]$,
\begin{align}\label{bound_first} 
\int_t^{t+s}  ( \alpha^{\veps_n} u^{\veps_n}(\sigma)^2  + \beta^{\veps_n}  {\veps_n}^{\eta} c_2^{\veps_n}(\sigma)) 
\left(\frac{1}{{\veps_n}} 
\int_{\Lambda_2^{\veps_n}}  \left\vert  \psi(x) \right\vert \, dx\right)\,d\sigma \leq K \| \psi|_{\infty} h\,.
\end{align}
We now focus on the second integral in the right hand side of \eqref{eq:equicontinuity_1}. Using upper bounds 
\eqref{linear_bound} and \eqref{bound_u}, we can find a constant $K$ such that for all $\veps >0$
\begin{multline*}
\int_t^{t+s} \int_0^{+\infty} \vert a^\veps(x) u^\veps(\sigma) f^\veps(\sigma,x) \Delta_{\veps}\psi(x)  -  
b^\veps(x) f^\veps(\sigma,x) \Delta_{-\veps}\psi(x) \vert\, dx\, d\sigma\\
\leq  K \| \psi'\|_{\infty} \int_t^{t+s}  \int_0^{+\infty} f^\veps(\sigma,x) (1+x)\, dxd\sigma\,.
\end{multline*}
Combining this last inequality with the moment estimate \eqref{eq:phi_1} and the inequality \eqref{bound_first},  there 
exists a constant $K$ (not depending on $\psi$, $h$ and $\veps$), such that for all $h \in(0,T)$, $t\in[0,T-h]$, 
 $s\in[0,h]$, $\psi\in\Cc_c^\infty([0,+\infty))$ and $n\geq0$
\begin{equation*} \label{eq:equi_h}
\left| \int_0^{+\infty} [f^{\varepsilon_n}(t+s,x)-f^{\varepsilon_n}(t,x) 
]\psi(x) \, 
dx \right|\leq K (\| \psi \|_{\infty} +\| \psi'\|_{\infty} )  h\,.
\end{equation*}
Let $\{\vphi_i\}_{i\geq 1} \subset\Cc_c^\infty([0,+\infty))$ a dense subset of $\Cc_c([0,+\infty))$ for the uniform 
norm. The metric $d$ defined by, for all $\mu$ and $\nu$ belonging to $\Mc_f([0,+\infty))$, 
\[d(\mu,\nu) = \sum_{i} \frac{2^{-i}}{\|\vphi_i\|_\infty+\|\vphi_i'\|_\infty} \left|\int_0^\infty \vphi_i \mu - 
\int_0^\infty \vphi_i 
\nu\right|\,,\]
is equivalent to the $weak-*$ topology (on bounded subset), see for instance similar construction in 
\cite[Theorem III.25]{Brezis}. Thus, for all $h \geq 0 \in(0,T)$, we have
\begin{equation*}
\sup_{t\in[0,T-h]}\, \sup_{s\in[0,h]}\, \sup_{n\geq0} \ d(f^{\veps_n}(t+s),f^{\veps_n}(t)) \leq  K h \,.
\end{equation*}
 This concludes the proof.
%
%
%
\end{proof}

\subsection{Compactness and limit} 

Here we give some technical lemmas which prepare the proof of the main results. 

\medskip

\begin{lem} \label{lem:weak_form_tx_eps}
For all $T>0$ and all $\vphi\in \Cc^1_c([0,T)\times[0,+\infty))$, we have, for all $\veps>0$,
\begin{multline}\label{eq:weak_form_tx_eps}
\int_0^T \int_0^{+\infty} \left[ \partial_t \vphi(t,x) + a^\veps(x) u^\veps(s)  \Delta_{\veps}\vphi(t,x)  -   
b^\veps(x) \Delta_{-\veps}\vphi(t,x) \right] f^\veps(t,x)\, dx\, dt \\
+ \int_0^{+\infty} f^{in,\veps}(x) \vphi(0,x)\, dx +  
\int_0^T [ 
 \alpha^\veps u^\veps(t)^2  -  \beta^\veps \veps^{\eta} c_2^\veps(t)] \left(\frac{1}{\veps} \int_{\Lambda_2^\veps}  
 \vphi(t,x)\, dx\right) dt = 0
\end{multline}
where $\Delta_h \vphi(t,x) = (\vphi(t,x+h) - \vphi(t,x))/h$, for $h\in\Rb$, and
\begin{equation} \label{eq:massss}
 u^\veps(t) + \int_0^\infty x f^\veps(t,x) \, dx = m^\veps.
\end{equation}
\end{lem}

\medskip

\begin{proof}
 The proof remains on multiplying each equation of the Becker-D\"oring system \ref{sys:BD_rescaled} by $\vphi_i = 
\int_{\Lambda_i^\veps} \vphi(t,x)\, dx$ for $\vphi\in \Cc^1_c([0,T)\times[0,+\infty))$ and using the definition of 
$f^\veps$ in Eq.~\eqref{eq:def_feps}. It is similar to Proposition \ref{prop:weak_eps}.
\end{proof}

\medskip

\begin{lem}\label{lem:compactness_weak_star}
Let $T>0$. The family $\{f^{\veps}\}$ is relatively $weak-*$ compact in  $L^\infty(0,T;\Mc_f([0,+\infty))$. If $\mu$ is 
an accumulation point of $\{f^{\veps}\}$, then there exists a sequence $\{\veps_n\}$ converging to $0$ and a 
non-negative function
$u\in C([0,T])$ such that $u^{\veps_n}$ converges to $u$ uniformly on $[0,T]$, with $u(0)=u^{\rm in}$ and   
\[  u(t) + \int_0^\infty x \mu_t(dx) = m.\]
Moreover, for all  $\vphi\in \Cc^1_c([0,T)\times[0,+\infty))$
\begin{multline*}
 \int_0^T \int_0^{+\infty} \left[ \partial_t \vphi(t,x) + a^{\veps_n}(x) u^{\veps_n}(s)  \Delta_{{\veps_n}}\vphi(t,x)  
- b^{\veps_n}(x) \Delta_{-{\veps_n}}\vphi(t,x) \right] f^{\veps_n}(t,x)\, dx\, dt \\
\to \int_0^T \int_0^{+\infty} \left[ \partial_t \vphi(t,x) + (a(x) u(s) -   
b(x)) \partial_x\vphi(t,x) \right] \mu_t(dx)\, dx
\end{multline*}
\begin{equation*}
\int_0^T  \alpha^{\veps_n} u^{\veps_n}(t)^2  \left(\frac{1}{{\veps_n}} \int_{\Lambda_2^{\veps_n}}  
 \vphi(t,x)\, dx\right) dt \to  \int_0^T  \alpha u(t)^2 \varphi(t,0) dt \,,
\end{equation*}
and 
\begin{equation*}
 \int_0^{+\infty} \vphi(0,x) f^{in,{\veps_n}}(x) \, dx \to \int_0^{+\infty} \vphi(0,x)\, \mu^{\rm in}(dx)\,
\end{equation*}
as $n\to +\infty$.
\end{lem}

\medskip

\begin{proof}
 First, remark the bound against $1$ in \eqref{eq:phi_1} yields to the relative compactness in 
$L^\infty(0,T;\Mc_f([0,+\infty))$. Let $\mu$ an accumulation point. By Lemma \ref{equicontinuity_u} and bound 
\eqref{bound_u} with Arzel\'a-Ascoli Theorem, entails there exists a sequence $\{\veps_n\}$ and $u\in\Cc([0,T])$ such 
that $u^{\veps_n}$ converge to $u$ uniformly on $[0,T]$ and $\{f^{\veps_n}\}$ to $\mu$. It remains to note that for 
any $\psi^\veps \in \Cc_c([0,T)\times[0,+\infty))$ which converge uniformly to some $\psi$, we have
\[ \int_0^T\int_0^\infty \psi^{\veps_n}(t,x) f^{\veps_n}(t,x)\, dx dt \to \int_0^T\int_0^\infty \psi(t,x) \mu_t(dx) 
dt\,,\]
as $n\to \infty$, to obtain the desired limit, see also \cite{Laurencot2002a,Collet2002}. In fact, using similar arguments as 
in Lemma \ref{lem:equicontinuity} with function in $\Cc_c((0,+\infty))$, we can obtain equicontinuity in
$\Mc_f((0,+\infty))$ for the $weak-*$ topology (open in $x=0$). Such result has been obtained for instance in 
\cite{Collet2002}. Thus, we could improve the compactness of $f^{\veps_n}$ in $\Cc([0,T];\Mc_f((0,+\infty))$ by Arzel\'a-Ascoli 
Theorem. Finally we obtain Eq. \eqref{eq:massss}, using the bound \eqref{eq:phi_1} with $\Phi$, and after  
regularization of the identity function, we have for 
all $t\in[0,T]$
\[\int_0^\infty x  f^{\veps_n}(t,x)\, dx \to \int_0^\infty x  \mu_t( dx)\,. \]
See \cite[Proof of Theorem 2.3]{Collet2002} for details.
\end{proof}

\medskip

\begin{lem} \label{lem:compactness_weak_star_2}
Assume $\eta\geq r_a$ and let a sequence $\{\veps_n\}$ converging to $0$. There exists $T>0$ and 
a sub-sequence $\{\veps_{n'}\}$ of $\{\veps_n\}$ such that  $\{f^{\veps_{n'}}\}$ is relatively compact in 
$\Cc([0,T];w-*-\Mc_f([0,+\infty))$ and $u^{\veps_{n'}}$ converge to some $u$ uniformly on $[0,T]$ with 
$\inf_{t\in[0,T]} u(t)>\rho$.
\end{lem}

\begin{proof}
Let $\widetilde T>0$ and $\{\veps_n\}$ a sequence converging to $0$.  Thanks to 
Lemma \ref{equicontinuity_u} and the bound \eqref{bound_u} we apply Arzel\'a-Ascoli Theorem, and there exists 
$u\in\Cc([0,\widetilde T])$ and a sub-sequence still denoted by  $\{\veps_n\}$ such that $u^\veps$ converge uniformly 
to $u$ on $[0,\widetilde T]$. By Assumption \ref{hyp:initial} we have $u(0)>\rho$, thus there exists $T\in 
(0,\widetilde T]$ such that we have $\inf_{t\in[0,T]} u(t)>\rho$. We can apply Lemma 
\ref{lem:equicontinuity} so that  $\{f^{\veps_n}\}$ is 
equicontinuous in $\Mc_f([0,+\infty))$. By the bound \eqref{eq:phi_1} (against $1$), we have for each $t\in[0,T]$ that 
$\{ f^{\veps_n}(t) \, : \, \veps>0 \}$ belongs to a $weak-*$ compact set of $\Mc_f([0,+\infty))$. Thus, again by 
Arzel\'a-Ascoli Theorem, the sequence $\{f^{\veps_n}\}$ is relatively compact in $\Cc([0,T];w-*-\Mc_f([0,+\infty))$.
\end{proof}

\medskip

\begin{remk}
 Convergence in $\Cc([0,T];w-*-\Mc_f([0,+\infty))$ entails convergence in $L^\infty(0,T;\Mc_f([0,+\infty))$ for the 
$weak-*$ topology.
\end{remk}

\section{Identification of the boundary term}\label{sec:identif}

This section is devoted to the proof of Theorems \ref{thm:LS_slow} to \ref{thm:LS_fast}. In view of Lemmas 
\ref{lem:weak_form_tx_eps} to \ref{lem:compactness_weak_star_2} it remains to identify the limit of $\veps^\eta 
c_2^\veps$ so that we can pass to the limit in the term
\[\int_0^T \beta^\veps \veps^{\eta} c_2^\veps(t) \left(\frac{1}{\veps} \int_{\Lambda_2^\veps}  
 \vphi(t,x)\, dx\right) dt \]
arising in \eqref{eq:weak_form_tx_eps}.

\medskip

We separate the following in 3 subsections corresponding to the 3 theorems. Thanks to Proposition 
\ref{prop:bound_laplace}, the compactness of the term $\veps^\eta c_2^\veps$ has been already obtained in 
$w-*-L^\infty(0,T)$ for the first two case, that are $\eta > r_a$ and $\eta = r_a$, and in $\Mc_f([0,T])$ by Eq. 
\eqref{bound_c_2_measure} for $\eta<r_a$. The identification of the limit relies on arguments similar to the 
Fenichel-Tikhonov theory on singularly perturbed dynamical systems \cite{kuehn_multiple_2015}. Multiplying the 
re-scaled BD equations \eqref{sys:BD_rescaled} by $\veps$, at least formally, we have for all $t>0$ and $i\geq 2$,
\[ \lim_{\veps \to 0} \veps \frac d {dt} c_i^\veps = \lim_{\veps \to 0} (J_{i-1}^\veps(t) - J_i^\veps(t) ) = 0\,.\]
Hence, at each time $t>0$, the underlying BD model for the discrete sizes $i\geq 2$ has to reach instantaneously the 
equilibrium of the BD model with a constant monomer concentration $u=u(t)$. Such version of the BD model has been well 
studied in \cite{Penrose1989,WATTIS}.

\subsection{Proof of Theorem \ref{thm:LS_slow} -- The slow de-nucleation case}
Let $\{\veps_n\}$ a sequence converging to $0$. By Lemma \ref{lem:compactness_weak_star_2}, there 
exists $T>0$, a sub-sequence, still denoted by  $\{\veps_n\}$ for simplicity, $\mu\in\Cc([0,T];w-*-\Mc_f([0,+\infty)))$ and $u\in\Cc([0,T])$ with
$\inf_{t\in[0,T]} u(t) >\rho$ such that $\{f^{\veps_{n}}\}$ converges to $\mu$ in $\Cc([0,T];\Mc_f([0,+\infty)))$ and $u^{\veps_n}$ converges to $u$ uniformly on $[0,T]$. Now, applying Proposition \ref{prop:bound_laplace}, we get
\[ \sup_{t\in[0,T]} \veps_{n}^\eta c_2^{\veps_{n}} (t) = \veps_{n}^{\eta-r_a} \sup_{t\in[0,T]} \veps^{r_a} c_2^{\veps_{n}} (t) \to 0 \,,\]
as $\eta>r_a$. Thus, combining this result with Lemma \ref{lem:compactness_weak_star} we can pass to the limit in 
\eqref{eq:weak_form_tx_eps} to obtain Eq.~\eqref{eq:weak_LS_fast2} with $N(u)=\alpha u^2$, and Theorem \ref{thm:LS_slow} is proved.

\subsection{Proof of Theorem \ref{thm:LS_compensated} -- The compensated nucleation case}
Let $\{\veps_n\}$ a sequence converging to $0$. We proceed similarly as above with Lemma~\ref{lem:compactness_weak_star_2} and Proposition~\ref{prop:bound_laplace}. As for all $i\geq 2$, $d_i^{\veps_{n}}=\veps_n^{r_a}c_i^{\veps_n}$ satisfies $d_i^{\veps_{n}} e^{-iz}\leq F^{\veps_{n}}(t,z)$, thanks to the estimate \eqref{bound_ci_caseII}, there exists $z>0$ such that
\[ \sup_{n\geq0} \, \sup_{t\in[0,T]} \, \sup_{i\geq 2} \, d_i^{\veps_{n}}e^{-iz} < +\infty\,.\]
Hence, by a Cantor diagonal process, we can extract another sub-sequence, still denoted by $\{\veps_n\}$, such that 
for all $i\geq 2$,
\[ d_i^{\veps_{n}}  \rightharpoonup d_i \,, \quad w-*-L^\infty(0,T)\,,\]
and 
\begin{equation}\label{eq:bound_di_ra}
0\leq \sup_{t\in[0,T]}\, \sup_{i\geq2} \, d_i(t) e^{-iz}< K_z \,.
\end{equation}
We recall, from the rescaled  BD system \eqref{sys:BD_rescaled}, that the sequence $(d_i^{\veps_{n}})_{i \geq 2}$ 
satisfies for each $n\geq 0$ Eq. \eqref{eq:BD_system_H}.
%
Hence, for all $\vphi\in \Cc^1([0,T])$,

\begin{multline}\label{eq:fast_case1}
{\veps_{n}}^{1-r_a}  d_{i}^{\veps_{n}}(t)\vphi(t) - {\veps_{n}}^{1-r_a}  d_{i}^{in,{\veps_{n}}} \vphi(0) - 
{\veps_{n}}^{1-r_a}\int_0^t  
d_i^{\veps_{n}}(s)\vphi'(s)\, ds \\
 = \int_0^t \vphi(s)  \left[ H_{i-1}^{\veps_{n}}(s) -  H_i^{\veps_{n}}(s) \right] \, ds\,.
\end{multline}
As $r_a<1$, passing to the limit ${\veps_{n}}\to 0$, the left hand-side in Eq.~\eqref{eq:fast_case1} vanishes, and, 
with Assumption \ref{hyp:coef_BD} on the kinetic rates, we have, for all $\vphi\in \Cc^1([0,T])$,
\[ \int_0^T \vphi(t)  \left[ H_{i-1}(t) -  H_i(t) \right] \, ds = 0 \,,\]
where $H_1 = \alpha u(t)^2 - \beta d_2 $, and for each $i\geq 2$,
\[ H_i = \begin{cases}
            \bar a i^\eta u d_i \,,& \text{if } \eta=r_a < r_b \,, \\[0.8em]
            \bar a i^\eta u d_i - \bar b (i+1)^\eta d_{i+1}\,, & \text{if } \eta=r_a=r_b \,. 
         \end{cases}
\]
Thus, for all $i\geq 2$, we have {\it a.e.} $t\in(0,T)$ that $H_i(t) = H_1(t)$. In the sequel, we will distinguish two 
cases, $r_a<r_b$ and $r_a=r_b$.

\subsubsection{The case $\eta=r_a<r_b$} In this case,  $H_1 = H_2$ for {\it a.e. } $t\in(0,T)$ yields
\[ d_2(t) = \frac{\alpha u^2(t)}{\overline a 2^\eta u(t)+\beta} \,.\]
Hence, the limit $d_2$ is uniquely identified (and by recurrence, all $d_i$, $i\geq 2$, using $H_i=H_1$) as a function 
of the limit $u$. Thus, combining this result with Lemma \ref{lem:compactness_weak_star} we can pass 
to the limit in \eqref{eq:weak_form_tx_eps} to obtain Eq.~\eqref{eq:weak_LS_fast2} with $N(u)=\alpha u^2 \frac{  u}{ u+\beta/(\bar a 2^{\eta})}$, and the case $r_a<r_b$ in Theorem \ref{thm:LS_compensated} is proved.

\subsubsection{The case $\eta=r_a=r_b$} In this case, the limit $(d_i)_{i\geq 2}$ must satisfy $H_i\equiv H$, $i\geq 
1$, for a given constant $H$. We classically (in the study of the equilibrium states of BD equations \cite{Ball1986}) 
define $ Q_1=1$ and for all $i\geq 
2$, 
\begin{equation*}
 Q_i=\frac{\alpha}{\beta}\prod_{k=2}^{i-1}\frac{\overline{a}k^{r_a}}{\overline{b}(k+1)^{r_a}},\ i\geq 2\,.
\end{equation*}
The solutions that satisfy $H_i\equiv H$ for all $i\geq 1$, are given by, after some algebraic manipulation (see 
\cite[lemma 1]{Penrose1989}),
\begin{equation*}
 d_i=Q_i u^i\Big{(}1-H\frac{1}{\alpha u^2}-H\sum_{k=2}^{i-1}\frac{1}{\overline{a}k^{r_a}Q_ku^{k+1}}\Big{)}\,,\quad 
i\geq2\,.
\end{equation*}
%
Thus, for all $i\geq 2$,
\begin{equation*}
 d_i=\frac{\alpha 
u^2}{\beta}\frac{2^{r_a}}{i^{r_a}}\Big{(}\frac{\overline{a}u}{\overline{b}}\Big{)}^{i-2}\left[1-\frac{H}{\alpha 
u^2}\left(1+\frac{\beta}{2^{r_a}}\frac{1}{\overline{a}u-\overline{b}}\right)+\frac{H \beta}{\alpha u^2 
2^{r_a}}\frac{\left(\overline{b}/(\overline{a}u)\right)^{i-2}}{\overline{a}u-\overline{b}}\right]\,.
\end{equation*}
%
%
However, for $u(t)>\rho = \overline b / \overline a$, there exists a unique $H$ such that the bound 
\eqref{eq:bound_di_ra} is satisfied, given by
\begin{equation*}
 H= \frac{\alpha u^2}{\left(1+\frac{\beta}{2^\eta}\frac{1}{\overline{a}u-\overline{b}}\right)}= \frac{\alpha u^2 
(\overline{a}u-\overline{b})}{\overline{a}u+\frac{\beta}{2^\eta}-\overline{b}}\,.
\end{equation*}
For this value, we have {\it a.e.} $t\in[0,T]$
\begin{equation*}
 d_2(t)=\frac{\alpha u(t)^2}{2^\eta(\overline{a}u-\overline{b})+\beta}=\frac{\alpha 
u(t)^2}{\beta}\left[1-\frac{\overline{a}u-\overline{b}}{\overline{a}u-\overline{b}+\beta/2^\eta}\right]\,.
\end{equation*}
Hence, proceeding as before we recover the second part of Theorem \ref{thm:LS_compensated}.

\subsection{Proof of Theorem \ref{thm:LS_fast} -- The fast de-nucleation} In the case $\eta<r_a$ we have no $L^\infty$ 
bound over $\veps^\eta c_2^\veps$, and no equicontinuity property on $\{f^\veps\}$ in $\Mc_f([0,+\infty))$. 
Nevertheless, we can apply  Lemma \ref{lem:compactness_weak_star}. Thus, let  $\widetilde T>0$ and $\{\veps_n\}$ a 
sequence converging to $0$, there exists a sub-sequence of $\{\veps_n\}$ (not relabeled),  $\mu \in 
L^\infty([0, \widetilde T];\Mc_f([0,+\infty)))$ and $u\in \Cc([0,\widetilde T])$ such that $f^{\veps_n} 
\rightharpoonup\mu$ in $w-*-L^\infty([0,\widetilde T];\Mc_f([0,+\infty)))$ and $u^{\veps_n}$ converges uniformly to $u$ 
on $[0,\widetilde T]$. Since $u^{\rm in}>\rho$ by Assumption \ref{hyp:initial}, there exists $T\in(0,\widetilde T]$ 
such that $\inf_{t\in[0,T]} u(t)>\rho$. Moreover, by the bound \eqref{bound_c_2_measure} we can extract another 
sub-sequence of $\{\veps_n\}$ (not relabeled) such that $d_2^{\veps_{n}}:=\veps_{n}^\eta c_2^\veps$ converges to a non-negative 
finite measure $\Gamma_2$ on $[0,T]$, where the 
convergence holds in $\Mc_f([0,T])$ endowed with the $weak-*$ topology. Also, for all $\vphi \in \Cc^1([0,T])$, the 
equation \eqref{sys:BD_rescaled} for $i=2$ yields 
\begin{multline}\label{limit_mesure_gamma2_temp}
  {\veps_{n}}^{1-r_a}  {\veps_{n}}^{r_a} c_{2}^{\veps_{n}}(T)\vphi(T) - {\veps_{n}}^{1-r_a}  {\veps_{n}}^{r_a}  
c_{2}^{in,\veps_n} \vphi(0) -  
{\veps_{n}}^{1-r_a} \int_0^T \vphi'(t) {\veps_{n}}^{r_a} c_{2}^{\veps_{n}}(t) dt \\
= \int_0^T \vphi(t) [ \alpha^{\veps_{n}} u^{\veps_{n}}(t)^2 - \beta^{\veps_{n}} d_2^{\veps_{n}}(t) ]dt \\ - 
\int_0^T \vphi(t) [\overline 
a_2^{\veps_{n}} {\veps_{n}}^{r_a-\eta} 
u^{\veps_{n}}(t) d_2^{\veps_{n}}(t) - \overline b_3^{\veps_{n}} {\veps_{n}}^{r_b} c_3^{\veps_{n}}(t)] dt\,.
\end{multline}
By Proposition \ref{prop:bound_laplace}, ${\veps_{n}}^{r_a} c_2^{\veps_{n}}(t)$ is uniformly bounded with respect to 
both time $t\in[0,T]$ and $n$, so that the left hand side of Eq.~\eqref{limit_mesure_gamma2_temp} goes to $0$ as 
${\veps_{n}}\to0$. Hence, with the bound \eqref{bound_c_2_measure} and since $\eta < r_a$, we have 
\begin{equation} \label{limit_mesure_gamma2}
 \lim_{{\veps_{n}}\to 0} \int_0^T \vphi(t)  {\veps_{n}}^{r_b} c_3^{\veps_{n}}(t) dt  = \frac{1}{\overline b_3} 
\left(  \int_0^T \vphi(t) \beta 
\Gamma_2(dt) - \int_0^T \vphi(t) \alpha u(t)^2 dt\right)\,.
\end{equation}
Here again two cases have to be considered, $r_a<r_b$ and $r_a=r_b$.
\subsubsection{The case $r_a<r_b$} In this case we use again Proposition \ref{prop:bound_laplace} for the left 
hand-side of Eq.~\eqref{limit_mesure_gamma2} and use that $\veps^{r_b-r_a} \to 0$ as ${\veps_{n}}\to0$. Thus, we are 
led with the following equality in measure
\[ \Gamma_2(dt) = \frac \alpha \beta  u(t)^2 dt\,. \]
Thus, combining this result with Lemma \ref{lem:compactness_weak_star} we can pass 
to the limit in \eqref{eq:weak_form_tx_eps} and we obtain the first case of  Theorem \ref{thm:LS_fast}.

\subsubsection{The case $r_a=r_b$}
In this case, we use again the fact that by Proposition~\ref{prop:bound_laplace}, up to a sub-sequence of $\{\veps_n\}$ (not relabeled), for all 
$i\geq 2$, there exists $d_i \in L^\infty(0,T)$ and $z_0>0$ such that
\[ {\veps_{n}}^{r_b} c_i^{\veps_{n}} \rightharpoonup d_i \, \quad w-*-L^\infty(0,T)\,,\]
and for all $z<z_0$, there exists $K_z>0$ such that
\begin{equation}\label{eq:bound_di_ra_fastdenucl}
 0 \leq \sup_{t\in[0,T]}\, \sup_{i\geq 2} \, d_i(t) e^{-iz}< K_z \,.
\end{equation}
From Eq.~\eqref{limit_mesure_gamma2}, we obtain the equality in measure 
\[ \overline b_3 d_3 \, dt  =    \beta \Gamma_2(dt) -  \alpha u(t)^2 \, dt \,.\]
Then, iterating the procedure, from equation \eqref{sys:BD_rescaled}, we get that, for all $i\geq 3$ and  $\vphi \in 
\Cc^1([0,T])$
\begin{multline*}
  {\veps_{n}}^{1-r_a}  {\veps_{n}}^{r_a} c_{i}^{\veps_{n}}(T)\vphi(T) - {\veps_{n}}^{1-r_a}  {\veps_{n}}^{r_a}  
c_{i}^{in,{\veps_{n}}} \vphi(0) -  
{\veps_{n}}^{1-r_a} \int_0^T \vphi'(t){\veps_{n}}^{r_a} c_i^{\veps_{n}}(t) dt \\
= \int_0^T \vphi(t) [ \overline a_{i-1}^{\veps_{n}}  u^{\veps_{n}}(t){\veps_{n}}^{r_a}c_{i-1}^{\veps_{n}}(t) - 
\overline b_i^{\veps_{n}} {\veps_{n}}^{r_a} 
c_i^{\veps_{n}}(t)] \, dt \\
- \int_0^T \vphi(t) [\overline a_i^{\veps_{n}} u^{\veps_{n}}(t) {\veps_{n}}^{r_a} c_i^{\veps_{n}}(t) - 
\overline b_{i+1}^{\veps_{n}} {\veps_{n}}^{r_a} 
c_{i+1}^{\veps_{n}}(t)] dt\,.
\end{multline*}
Hence, for $i=3$, writing $ {\veps_{n}}^{r_a}c_{2}^{\veps_{n}}(t)= \veps_n^{r_a-\eta} d_{2}^{\veps_{n}}(t)\to 0$ (in 
$\Mc_f([0,T])$), we 
obtain
\begin{equation*}
 0= \int_0^T \vphi(t)[ - \overline b_3 d_3(t) - \overline a_3 u(t) d_3(t) + \overline b_{4}d_{4}(t)] dt\,.
\end{equation*}
And for all $i\geq 4$,
\begin{equation*}
 0= \int_0^T \vphi(t)[\overline a_{i-1} u(t) d_{i-1}(t) - \overline b_i d_i(t) - \overline a_i u(t) d_i(t) + \overline 
b_{i+1}d_{i+1}(t)] dt\,.
\end{equation*}
With $H_2=-\overline b_3 d_3$, $H_i=\overline a_i u^\veps d_i(t) - \overline b_{i+1}d_{i+1}$, $i\geq 3$, then we must 
have a.e. $H_i=H_2=:H$, for all $i\geq 2$. Then we get, for all $i\geq 3$,
\begin{equation*}
 d_i(t)=-\frac{H}{\overline{b}_i}\sum_{j=3}^i\left( \prod_{k=j}^{i-1} 
\frac{\overline{a}_k}{\overline{b}_{k}}\right)u^{(i-j)}=-\frac{H}{\overline{b}_i}\sum_{j=3}^i
 \left(\frac{\overline{a}u}{\overline{b}}\right)^{i-j}\,.
\end{equation*}
In order to fulfil the bound \eqref{eq:bound_di_ra_fastdenucl},
we must get $H=0$, so that $d_3=0$ and the following equality in measure holds
\[   \Gamma_2(dt) = \frac \alpha \beta  u(t)^2 dt \,.\]
This ends the proof of Theorem \ref{thm:LS_fast}.

\section{Extension to a density}\label{sec:density}

In this section, we make an extra-assumption in order to obtain a convergence result in $L^1$ functional space, so that the limit measure has a density with respect to the Lebesgue measure:

\begin{hyp}\label{assumption_5}
There is $\delta\in(0,1/r_a-1)$ such that, for the function $\Psi(y)=y^{1+\delta}$, 
\begin{equation}\label{hyp:psi_init} \tag{H8}
\sup_{\veps>0}\int_0^{\infty}\Psi(f^{in,\veps}(x))dx<\infty\,.
\end{equation}
Moreover, the kinetic rates are given by exact power law functions, \textit{i.e.},
 \begin{equation}\label{hyp:powerlow} \tag{H9}
  \begin{array}{ll}
   \ds a_i^\veps =  \overline{a}(\veps i)^{r_a}\,,& i\geq 2\,, \\[0.8em]
   \ds b_i^\veps = \overline{b}(\veps i)^{r_b}\,, & i\geq 3\,. 
  \end{array}
 \end{equation}
\end{hyp}
\begin{remk}
The first hypothesis \eqref{hyp:psi_init} is slightly stronger than a compactness hypothesis in $L^1(dx)$, where a more 
general (and not explicit) $\Psi$ can be obtained, see \cite{Chauhoan1977}. However, having an explicit power low 
function for $\Psi$ will simplify the following calculus. The same is valid for the extra 
hypothesis \eqref{hyp:powerlow} on the kinetic rates (which is in agreement with hypothesis \eqref{H5}). 
\end{remk}

\medskip

\noindent Assuming Assumption \ref{assumption_1}-\ref{assumption_5} hold true, we can now prove the last result.

\medskip

\begin{thm}\label{thm_density}
Assume $\eta \geq r_a$ and $r_a=r_b$. Let a sequence $\{\veps_n\}$ converging to $0$. There 
exists $T>0$,  a sub-sequence 
$\{\veps_{n'}\}$ of $\{\veps_n\}$, and $f\in \Cc([0,T],w-L^1(\Rb_+,x^{r_a\delta}dx))\cap 
L^\infty(0,T;L^1(\Rb_+,(1+x)dx))$ 
such that the measure $f(t,x)dx$ is a $N$-solution of LS with mass $m$ and
\[f^{\veps_{n'}} \xrightharpoonup[n'\to+\infty]{} f\]
in $\Cc([0,T];w-L^1(\Rb_+,x^{r_a\delta}dx))$. $N$ is given in Theorem 
\ref{thm:LS_slow}-\ref{thm:LS_compensated} according to the value of $\eta$.
\end{thm}

\medskip

\noindent The proof of this theorem is based on the following lemma which proof is postponed below

\medskip

\begin{lem} \label{lem:psi_bound}
Assume $\eta \geq r_a$ and $r_a=r_b$.  Let a sequence $\{\veps_n\}$ converging to $0$. There exist $T>0$ 
and a sub-sequence $\{\veps_{n'}\}$ of $\{\veps_n\}$ such that 

\begin{equation*}
 \sup_{n'\geq0} \, \sup_{t\in[0,T]} \, \int_0^\infty \min(1,x^{r_a\delta}) \Psi(f^{\veps_{n'}}(t,x))\, dx 
< +\infty\,.
\end{equation*}

\end{lem}

\medskip

{ {\it Proof of Theorem \ref{thm_density}.}
We reproduce the same proof as for Theorem \ref{thm:LS_slow} and \ref{thm:LS_compensated} and obtain a sub-sequence $f^{\veps_{n'}}$
that converges in measure. We now remark that, combining the estimates \eqref{eq:phi_1} in Lemma \ref{lem:u1xphi1} 
and the last Lemma \ref{lem:psi_bound} we can apply the Dunford-Pettis theorem and we have a weak compact subset 
$\mathcal K$ of $L^1(\Rb_+,x^{r\delta}dx)$ such that for all $t\in[0,T]$ and $n'\geq0$, $f^{\veps_{n'}}(t)\in\mathcal 
K$. We are 
now in position to prove that along another subsequence, still denoted by $\{\veps_{n'}\}$, the 
sequence converges to some $f$ in $C([0,T],w-L^1(\Rb_+,x^{r\delta}dx))$. Moreover, $f$ belongs to 
$L^\infty(0,T,L^1(\Rb_+,(1+x)dx)$. 
The proof follows similar arguments as in \cite[Proof of Theorem 2.2, p. 981]{Laurencot2002a} which consists in proving 
the equicontinuity of 
\[t \to \int_{0}^R f^\veps(t,x)\vphi(x)x^{r\delta}dx\,,\]
for all $\vphi\in L^\infty(0,R)$ and $R>0$. Indeed, by Eq.~\eqref{eq:equi_h} we have for any $\vphi \in \Cc^1$ with 
compact support in $(0,R)$ that (see also the proof of lemma \ref{lem:equicontinuity})
\[ \lim_{h\to0} \, \sup_{t\in[0,T-h]} \, \sup_{s\in(0,h)}\, \left| \int_0^\infty 
(f^\veps(t+s,x)-f^\veps(t,x) )\vphi(x)\, x^{r\delta} dx \right| = 0 \,. \]
Then taking a pointwise convergent sequence $\{\vphi^n\}$ in $\Cc_c([0,R])$ of $\vphi\in L^\infty(0,R)$ and using 
Egorov's theorem we get the desire results. Finally, we apply a variant of Arzela-Ascoli theorem for weak topology, 
see \cite[Theorem 1.3.2]{Vrabie}, so that for each $R>0$, the sequence is relatively compact in 
$C([0,T],w-L^1((0,R),r^{r\delta}dx)$. 
By the compact containment  we improve this results on $\Rb_+$.

\subsection*{Technical results}

\medskip

\noindent Before proving Lemma \ref{lem:psi_bound}, we start by some technical lemmas.

\medskip

\begin{lem}\label{lem1}
Let $\vphi \in \Cc_b\left([0,\infty)\right)$ non-negative. Then, for any $I\geq 3$,
\begin{multline} \label{eq:vphi_Psi}
\int_0^\infty \vphi(x) \left[ \Psi(f^\veps(t,x)) - \Psi(f^{in,\veps}(x)) \right]  \, dx \\ \leq \veps \sum_{i=2}^{I-1} 
\vphi_i^\veps \Psi(c_i^\veps(t)) + \int_0^t \left[\vphi_{I}^\veps a_{I-1}^\veps u^\veps(s) \Psi(c_{I-1}^\veps(s))  - 
\vphi_{I-1}^\veps b_I^\veps \Psi(c_{I}^\veps(s))\right]ds \\
+\int_0^t\int_{(I-1/2)\veps}^\infty \Big{[} a^\veps(x) u^\veps(s) \Delta_\veps\vphi(x) - b^\veps(x)   
\Delta_{-\veps}\vphi(x) \\ - \delta \left( u^\veps(s) \Delta_{-\veps}  a^\veps(x) - \Delta_{\veps}  b^\veps(x)  
\vphi(x)\right) \Big{]}\Psi(f^\veps(x,s))dxds \,. 
 \end{multline}
 where  $\vphi_i^\veps =  1/{\veps}  \int_{\Lambda_i^\veps} \vphi(x) \, dx$.
\end{lem}

\medskip

\begin{proof}
The proof follows similar lines as in \cite[Lemma 4.1]{Laurencot2002a}, but we take profit of the 
explicit form of $\Psi$ to obtain a necessary finer estimate. We sketch it briefly below. From the (BD) system 
\eqref{sys:BD_rescaled}, it comes
\begin{multline*}
\int_0^\infty \vphi(x) \left[ \Psi(f^\veps(t,x)) - \Psi(f^{in,\veps}(x)) \right]  \, dx= \sum_{i\geq 2} 
\int_{\Lambda_i^\veps} \vphi(x)  \left[\Psi(c_i^\veps(t)) - \Psi(c_i^\veps(0))\right] \, dx \\
=  \veps \sum_{2\leq i\leq I-1}   \vphi_i^\veps \left[\Psi(c_i^\veps(t)) - \Psi(c_i^\veps(0))\right] + \sum_{i\geq I} \vphi_i^\veps \int_0^t [J_{i-1}^\veps(s) - 
J_i^\veps(s)] \Psi'(c_i^\veps(s))\, ds \,.
\end{multline*}
 We can decompose the latter in three parts,
 
\begin{equation*} 
 \begin{array}{rcl}  
   \ds \int_0^\infty \vphi(x) \left[ \Psi(f^\veps(t,x)) - \Psi(f^{in,\veps}(x)) \right] \, dx  & = & \ds  N^\veps(t) 
+ \int_0^t [A^\veps(s)+B^\veps(s)] \, ds \,,
 \end{array}
\end{equation*}
where
\begin{equation*}
\begin{array}{rcl}
\ds N^\veps(t) & := & \ds \veps   \sum_{2\leq i\leq I-1} \vphi_i^\veps  \left[\Psi(c_i^\veps(t)) - \Psi(c_i^\veps(0))\right]\,, \\[0.8em]
\ds A^\veps(t) & := & \ds \sum_{i\geq I} \vphi_i^\veps u^\veps(t) [a_{i-1}^\veps c_{i-1}^\veps(t) - a_{i}^\veps 
c_{i}^\veps(t)] \Psi'(c_i^\veps(t))\,,\\[0.8em]
\ds B^\veps(t) & := & \ds \sum_{i\geq I} \vphi_i^\veps [ b_{i+1}^\veps c_{i+1}^\veps(t)-b_{i}^\veps c_{i}^\veps (t)]  
\Psi'(c_i^\veps (t))\,.
\end{array}
\end{equation*}
Then, in $A^\veps$ we can re-write, using the convexity of $\Psi$, for all $i\geq I$,
\begin{multline*}
[a_{i-1}^\veps c_{i-1}^\veps(t) - a_{i}^\veps c_{i}^\veps(t)] \Psi'(c_i^\veps(t)) \\
=  a_{i-1}^\veps [c_{i-1}^\veps(t) -  c_{i}^\veps(t)]  \Psi'(c_i^\veps(t))  + (a_{i-1}^\veps-a_i^\veps) 
c_{i}\Psi'(c_i^\veps(t)) \\
\leq a_{i-1}^\veps \left( \Psi(c_{i-1}^\veps(t)) -   \Psi(c_{i}^\veps(t)) \right) +(a_{i-1}^\veps-a_i^\veps) 
c_i^\veps(t)\Psi'(c_i^\veps(t)) \,.
\end{multline*}
Then, reordering the term in the last inequality and then using that $x\Psi'(x) - \Psi(x) = \delta \Psi(x)$,
\begin{multline*}
[a_{i-1}^\veps c_{i-1}^\veps(t) - a_{i}^\veps c_{i}^\veps(t)] \Psi'(c_i^\veps(t))\\
\leq a_{i-1}^\veps  \Psi(c_{i-1}^\veps(t)) -  a_i^\veps \Psi(c_{i}^\veps(t)) +(a_{i-1}^\veps-a_i^\veps)  [ 
c_i^\veps(t)\Psi'(c_i^\veps(t)) 
-\Psi(c_i^\veps(t))] \\
= a_{i-1}^\veps  \Psi(c_{i-1}^\veps(t)) -  a_i^\veps \Psi(c_{i}^\veps(t)) - \delta  (a_{i}^\veps-a_{i-1}^\veps) 
\Psi(c_i^\veps(t))\,.
\end{multline*}
Thus, we obtain for $A$ the following estimation,
\begin{multline*}
 A^\veps(t) \leq \sum_{i\geq I}  a_i^\veps u^\veps (\vphi_{i+1}^\veps-\vphi_i^\veps) \Psi(c_i^\veps(t)) + 
\vphi_{I}^\veps a_{I-1}^\veps u^\veps(t) 
\Psi(c_{I-1}^\veps(t))\\- \delta u^\veps \sum_{i\geq I}\vphi_i^\veps (a_i^\veps-a_{i-1}^\veps) \Psi(c_i^\veps) \,.
\end{multline*}
We estimate $B$, by similar argument, to get,
\begin{multline*}
 B^\veps(t) \leq \sum_{i\geq I} \vphi_i^\veps [b_{i+1}^\veps \Psi(c_{i+1}^\veps)- b_i^\veps \Psi(c_i^\veps)] +\delta \sum_{i\geq I} \vphi_i^\veps (b_{i+1}^\veps -b_i^\veps) \Psi(c_i^\veps) \\
 \leq  \sum_{i\geq I} (\vphi_{i-1}^\veps - \vphi_{i}^\veps) b_{i}^\veps \Psi(c_{i}^\veps)  - \vphi_{I-1}^\veps b_I^\veps \Psi(c_{I}^\veps) +\delta \sum_{i\geq I} \vphi_i^\veps (b_{i+1}^\veps -b_i^\veps) \Psi(c_i^\veps)\,. 
\end{multline*}
Both estimates on  $A^\veps$ and  $B^\veps$ directly give \eqref{eq:vphi_Psi}.
\end{proof}

\medskip

\begin{lem}\label{lem:estimequitue}
For all $0\leq r<1$, and for all $0<\delta<\frac{1}{r}-1$, there exists $I_0$ such that for all $i\geq I_0$, and all 
$x\in[0,1]$,  
\begin{equation*}
\Big{[}i^{r}\left((i+1/2+x)^{r\delta}-(i-1/2+x)^{r\delta}\right)-\delta(i^{r}-(i-1)^{r}) (i-1/2+x)^{r\delta}\Big{]}\leq 
0\,,
\end{equation*}
\end{lem}

\medskip

\begin{proof}
Doing an expansion as $i\to \infty$, we easily obtain
\begin{multline*}
\Big{[}i^{r}\left((i+1/2+x)^{r\delta}-(i-1/2+x)^{r\delta}\right)-\delta(i^{r}-(i-1)^{r}) (i-1/2+x)^{r\delta}\Big{]}\\
=r\delta \frac{i^r(i-\frac{1}{2}+x)^{r\delta}}{i^2}\Big{[}\frac{r(1+\delta)-1}{2}-x+O(\frac{1}{i})\Big{]}\,.
\end{multline*}
We conclude straightforwardly as $r(1+\delta)-1<0$.
\end{proof}

\medskip

\subsection*{Proof of Lemma \ref{lem:psi_bound}}
In the following, let $r=r_a=r_b$ and $I=I_0$ given by Lemma \ref{lem:estimequitue}. We want to 
bound each term of Eq. \eqref{eq:vphi_Psi} with  $\vphi(x)=\min(1,x^{r\delta})$.
Remark the term $-\vphi_{I_0-1}^\veps b_{I_0}^\veps \Psi(c_{I_0}^\veps(t))$ can be easily 
dropped in Eq. \eqref{eq:vphi_Psi} since it is non-positive. Also, note that, for $2\leq i\leq I_0$,
\begin{equation*}
\veps \vphi_i^\veps \Psi(c_i^\veps(t)) \leq \veps^{1-r(1+\delta )} \vphi_i^\veps 
\left(\veps^{r}c_i^\veps(t)\right)^{1+\delta}\,.
\end{equation*}
Thus, since $\vphi_i^\veps$ is bounded  and $\delta \leq 1/r-1$, we apply Lemma 
\ref{lem:compactness_weak_star_2} and Proposition \ref{prop:bound_laplace} to obtain $T>0$ and a sub-sequence, still denoted by $\{\veps_n\}$, such that
\begin{equation}\label{begin}
\sup_{n\geq 0}\, \sup_{t\in[0,T]} \, \left(\veps_n \vphi_i^{\veps_n} \Psi(c_i^{\veps_n}(t))\right)<\infty.
\end{equation}
Similarly, using that $u^\veps(t)\leq K_m$, we have
\begin{multline}
\vphi_{I_0}^\veps a_{I_0-1}^\veps u^\veps(t)\Psi(c_{I_0-1}^\veps(t)) = 
\overline{a}(I_0-1)^{r}u^\veps(t)\left(\int_{I_0-1/2}^{I_0+1/2}y^{r \delta}dy\right) \left(\veps^{r} 
c_{I_0-1}^\veps(t)\right)^{1+\delta}\\
 \leq K_m \overline{a}(I_0-1)^{r}\left(\int_{I_0-1/2}^{I_0+1/2}y^{r \delta}dy\right) 
\sup_{\veps>0} \, \sup_{t\in[0,T]}\, \left(\veps^{r} c_{I_0-1}^\veps(t)\right)^{1+\delta} <\infty\,,
\end{multline}
By these estimates, the boundary terms in Eq.~\eqref{eq:vphi_Psi} are uniformly bounded. We are lead 
with the remaining integral term on $\left((I_0-1/2)\veps,\infty\right)$. Denote, for all $\veps>0$ and $x>0$, 
\[ D^\veps(x) = a^\veps(x) u^\veps(t) \Delta_\veps\vphi(x)- b^\veps(x)   \Delta_{-\veps}\vphi(x)  
-\delta \left( u^\veps \Delta_{-\veps}  a^\veps(x) - \Delta_{\veps}  b^\veps(x) \right)\vphi(x)\,. \]
Thus,
\begin{multline*}
\int_{(I_0-1/2)\veps}^{1} D^\veps(x) \Psi(f^\veps(x,t))dx\\
  =\sum_{i= I_0}^{1/\veps} \frac{1}{\veps}\int_{\Lambda_i^\veps}\Big{[} \left( a_i^\veps u^\veps(t) 
(\vphi(x+\veps)-\vphi(x)) - b_i^\veps   (\vphi(x)-\vphi(x-\veps)) \right) \\-\delta \left( 
u^\veps(a_{i}^\veps-a_{i-1}^\veps) -  (b_{i+1}^\veps-b_i^\veps) \right)\vphi(x) \Big{]} \Psi(c_i^\veps(t))dx\,.
\end{multline*}
Then, on $x\in (0,1)$, we have that $\vphi(x)=x^{r\delta}$, and letting $\Gamma_i=[i-1/2,i+1/2)$ and changing variable $\veps y=x$, we obtain
\begin{multline*}
\int_{(I_0-1/2)\veps}^{1} D^\veps(x) \Psi(f^\veps(x,t))dx\\
= \sum_{i= I_0}^{1/\veps} \veps^{r(1+\delta)}\int_{\Gamma_i}\Big{[} \left( \overline{a} i^r u^\veps(t) 
 ((y+1)^{r\delta}-y^{r\delta}) - \overline{b}i^r   (y^{r\delta}-(y-1)^{r\delta}) \right)\\-\delta \left( u^\veps 
 \overline{a}(i^r-(i-1)^r) - \overline{b} ((i+1)^r-i^r) \right)y^{r\delta} \Big{]} \Psi(c_i^\veps(t))dy\,.
\end{multline*}
Finally, rearranging the term we have
\begin{multline*}
\int_{(I_0-1/2)\veps}^{1} D^\veps(x) \Psi(f^\veps(x,t))dx\\
=\sum_{i= I_0}^{1/\veps} \veps^{r(1+\delta)}\int_{\Gamma_i}\Big{[} \left( \overline{a}  u^\veps(t)- 
 \overline{b}\right) \left( i^r ((y+1)^{r\delta}-y^{r\delta}) -\delta (i^r-(i-1)^r)y^{r\delta} \right)  \\ 
+\overline{b}i^r \left((y+1)^{r\delta}-2y^{r\delta}+(y-1)^{r\delta}\right)\\
+\delta  \overline{b} \left((i+1)^r-2i^r 
+(i-1)^r\right) y^{r\delta} \Big{]} \Psi(c_i^\veps(t))dy\,.
\end{multline*}
Then, as the second discrete derivative are negative, that is, for all $s<1$ and all $x>1$,
\begin{equation*}
\left((x+1)^{s}-2x^{s}+(x-1)^{s}\right)\leq 0\,,
\end{equation*}
we obtain
\begin{multline*}
\int_{(I_0-1/2)\veps}^{1} D^\veps(x) \Psi(f^\veps(x,t))dx\\
\leq \veps^{r(1+\delta)} \left( \overline{a}  
u^\veps(t)- \overline{b}\right) \sum_{i= I_0}^{1/\veps} \int_{\Lambda_i}\Big{[} i^r ((y+1)^{r\delta}-y^{r\delta})\\ 
-\delta (i^r-(i-1)^r)y^{r\delta}  \Big{]} \Psi(c_i^\veps(t))dy\,.
\end{multline*}
The term under the integral is negative by Lemma \ref{lem:estimequitue}. We now fix $T>0$ and extract a 
sub-sequence $\{\veps_{n'}\}$ given by Lemma  \ref{lem:compactness_weak_star_2} such that $\overline{a}  
u^\veps(t)- \overline{b} > 0$ on $[0,T]$. Thus,
\begin{equation}
\int_{(I_0-1/2)\veps}^{1} D^\veps(x) \Psi(f^\veps(x,t))dx \leq 0 \,.
\end{equation}
On the other hand we have, since $\Delta_\veps \vphi = 0$ on $(1,+\infty)$,
\begin{multline} \label{final}
\int_{1}^{\infty} \Big{[} D^\veps(x) \Psi(f^\veps(x,t))dx\\ 
\leq  \delta ( K_m \sup_{x\geq 1}\mid a'(x) \mid  + \sup_{x\geq 1}\mid b'(x) \mid)  
\int_{1}^{\infty}  \vphi(x)  \Psi(f^\veps(x,t))dx\,,
\end{multline}
and we conclude by estimates \eqref{begin} to \eqref{final} that, for some constant $K>0$ and all $t\in[0,T]$, 
\begin{equation*}
\int_0^\infty \vphi(x) \Psi(f^{\veps_n}(t,x)) \leq  K + \int_0^\infty \Psi(f^{in,\veps_n}(x))\, dx 
 + K \int_0^t \int_0^\infty \vphi(x) \Psi(f^{\veps_n}(t,x))\,. 
\end{equation*}
We conclude the proof with the Gronwall Lemma.

\subsection*{The general case}
The main difficulty to treat the case $r_a<r_b$ is to find a test function $\vphi$ in Eq. \eqref{eq:vphi_Psi} which 
make the term under the integral negative around $0$, but which also keep the boundary terms bounded. We believe that a 
good function would be 
\[\vphi(x) = \min( x^{r\delta}e^{-Kx^{r_b-r_a}},c)\,,\]
for some $c>0$ small and $K>0$ large enough. It recovers the case $r_a=r_b$ (with $c=1$). Computations are not presented here because too fastidious. Just let us show that, at the limit $\veps\to 0$,
\begin{multline*} 
 [ \overline a x^{r_a} u(t)  - \overline b x^{r_b}] \vphi'(x) - \delta \left[ r_a \overline a x^{r_a-1} u(t) - 
r_b \overline b x^{r_b-1} \right] \vphi(x)  \\
 = \frac{\vphi(x)}{x}  (r_b-r_a)\left[ \delta \overline b x^{r_b} - K x^{r_b-r_a} (\overline a x^{r_a} u(t) -\overline 
b x^{r_b}) \right] \,.
\end{multline*}
But since $u(t)> \rho$, it exists $x_0>0$ small and $\gamma>0$ such that the flux is bounded from below by $ \overline 
a x^{r_a}u(t) - \overline b x^{r_b} \geq \gamma \overline a x^{r_a}$ on 
$[0,x_0]$, thus

\begin{multline*}
 [ \overline a x^{r_a} u(t)  - \overline b x^{r_b}] \vphi'(x) - \delta \left[ r_a \overline a x^{r_a-1} u(t) - 
 r_b \overline b x^{r_b-1} \right] \vphi(x)  \\
 \leq  \frac{\vphi(x)}{x}(r_b-r_a) \left[  \delta \overline b - K \gamma  \right] x^{r_b} \,.
\end{multline*}
Hence, for $K$ large enough the term is negative around $0$, which was the essential ingredient of the proof of Theorem \ref{thm_density}.

\section{Discussion}\label{sec:disc}

In this work, we  obtained limit theorems to derive rigorously the link between a discrete-size 
coagulation-fragmentation model, the Becker-D\"oring (BD) model, and a continuous-size model, the Lifshitz-Slyozov (LS) 
model. We used weak-convergence in measure, to prove that a sequence of discrete stepwise functions 
associated to the BD model converges towards a measure solution of the LS model. The novelty of our work, compared to 
previous work in \cite{Laurencot2002a,Collet2002}, consists of being able to rigorously defined a boundary flux 
condition for the limit non-linear transport partial differential equation of the LS model. This boundary condition has 
been obtained thanks to an averaging procedure for the smaller-sized cluster, namely the one of size $i=2$. It is 
classical when passing from a discrete to a continuous model (think of a random walk converging to a Brownian motion) 
to accelerate the rates (or equivalently, the time) between each discrete transition. Hence, each individual 
discrete-size cluster evolves in the re-scaled BD model~\eqref{sys:BD_rescaled} at a faster time scale than the 
continuous density function $f^\veps$ in Eq.~\eqref{eq:weak_form_eps}. Although the fast-motion involves a dynamical 
system of infinite dimension, we could obtain appropriate $L^\infty$-bounds on the time trajectories of each 
discrete-sized cluster, and proves that, in the limit when the scaling parameter $\veps\to 0$, each discrete-sized 
cluster is the unique solution of an algebraic equation, which appears to be the same as the steady-state condition of a constant monomer BD model.

 Let us now discuss in more details what were the scaling assumptions that lead to the study of the 
system~\eqref{sys:BD_rescaled} (for the mathematical derivation, see the appendix~\ref{annex:adimensionalixation}). 
Roughly, the system~\eqref{sys:BD_rescaled} is obtained when we consider that the clusters have very large sizes but 
are present in a low quantity compared to a large excess of free particles. The rescaled equations are obtained in a 
large volume hypothesis, and the scaling of the macroscopic reaction rates accounts for the volume-dependence of the 
aggregation (so that aggregation and fragmentation occur at the same time scale).

\noindent However, importantly enough, the first aggregation (nucleation) rate is scaled differently from the other 
aggregation rates  (see Appendix \ref{annex:adimensionalixation}) and this comes from the special role played by the 
free particles. 
Despite the large excess of free particles, in this framework, the nucleation occurs at the same time scale than the 
aggregation of large-sized clusters, and has for consequence to prevent the formation of too many clusters.  A 
different choice at this step would lead to a rapid depletion of free particles, and would result in different mass 
conservation where free particles are not present as a distinct entity any more-- see the work \cite{Laurencot2002a} on 
the Lifshitz-Slyozov-Wagner equation.

\noindent  Finally, we allowed a flexibility in the scaling of the first fragmentation (de-nucleation), quantified by 
the exponent $\eta$. We found (see Theorems \ref{thm:LS_slow}-\ref{thm:LS_compensated}-\ref{thm:LS_fast} ) that 
different values of $\eta$ give rise to distinct boundary condition at the limit when $\veps$ goes to $0$. The most 
natural case, $\eta=r_b$, corresponds to the case where the clusters of size $2$ dissociate at the same speed than the 
small-sized clusters of size $i$, $i\geq 3$. Then, the case $\eta>r_b$ corresponds to an asymptotically irreversible 
nucleation (and leads to a macroscopic flux $N(t)= \alpha u(t)^2$, which corresponds to the microscopic nucleation rate 
-- this conclusion actually holds for all $\eta>r_a$). And the case $\eta\leq r_a<r_b$ corresponds to a strongly 
reversible de-nucleation (and leads to $0 \leq N(t)<\alpha u(t)^2$ according to the value $r_a$).

 Hence, our work shed lights on which appropriate boundary condition should be used for the LS equation  (or 
similar continuous coagulation models) according to specific microscopic hypotheses (unfavorable, balanced or 
irreversible nucleation). We believe that our procedure could be applied to several related models (for instance, the  
Lifshitz-Slyozov-Wagner equation mentioned above, or the prion equation \cite{doumic}) and should help to build reduced structured population models while 
taking into account of their intrinsic multi-scale nature (see \cite{Yvinec2012,Yvinec2016} for applications).

\appendix

\section{From the original to the dimensionless BD system} \label{annex:adimensionalixation}

The original BD model gives the evolution of $(c_i)_{i\geq 1}$ by 
\begin{equation*}
 \begin{array}{rcll}
  \ds \frac{d}{dt}c_1 & \ds = & \ds -J_1-\sum_{i=1}^\infty J_i \, , & t\geq 0\,, \\[1.5em] 
  \ds \frac{d}{dt}c_i & \ds = & \ds J_{i-1}- J_i \, , & t\geq0\,, \ i\geq 2\,, 
 \end{array}
\end{equation*}
where $J_i$ is the flux between clusters of size $i$ and $i+1$, given by
\begin{equation*}
 J_i=a_ic_1c_i-b_{i+1}c_{i+1}\,,\quad i\geq 1\,.
\end{equation*}
Here, coefficients $a_i$ and $b_{i+1}$ denote respectively the rate of aggregation and the rate of fragmentation. 
Observe that such model (at least formally) preserves the total number of particles (no source nor sink), that is
\begin{equation*}
 \sum_{i=1}^\infty ic_i(t)= \sum_{i=1}^\infty ic_i(0)=: m \,, \quad  t\geq0\,.
\end{equation*}
The classical approach to operate a scaling is to write the equations in a dimensionless form. We follow 
\cite{Collet2002} and introduce the following characteristic values:
\begin{itemize}
 \item[$\overline T\phantom{_1}$] : characteristic time,
 \item[$\overline C_1$] : characteristic value for the free particle concentration $c_1$ ,
 \item[$\overline C\phantom{_1}$] : characteristic value for the cluster concentration $c_i$, for $i \geq 2$,
 \item[$\overline A_1$] : characteristic value for the first aggregation coefficient $a_1$,
 \item[$\overline B_2$] : characteristic value for the first fragmentation coefficient $b_2$,
 \item[$\overline A\phantom{_1}$]: characteristic value for the aggregation coefficients $a_i$, $i\geq 2$,
 \item[$\overline B\phantom{_1}$]: characteristic value for the fragmentation coefficients $b_i$, $i\geq3$, 
 \item[$\overline M_c$] : characteristic value for the total mass $m$.
\end{itemize}
Thus, the dimensionless quantities are 
\[ \tilde t = t/\overline T \,, \quad \tilde m = m/\overline M_c\,, \quad \tilde u( \tilde t) = c_1( \tilde t \overline 
T)/\overline C_1, \quad \tilde c_i( \tilde t) = c_i( \tilde t \overline T)/\overline C\,, \]
and for all $i\geq 2$,
\begin{equation*}
\tilde a_i  =  a_i/\overline A\,, \quad  \tilde b_{i+1}  =   b_{i+1}/\overline B\,, 
\end{equation*}
and the particular scaling at the boundary (we use different letters to emphasize this point):
\[\tilde \alpha  :=  a_1/ \overline A_1\,, \quad \tilde \beta :=   b_2/ \overline B_2\,. \]
Then, the quantities $\tilde u(\tilde t)$, $\tilde c_i(\tilde t)$ satisfy the equation 

\begin{equation*}
\begin{array}{l}
\ds \frac{d}{d\tilde t} \tilde u = \ds \frac{\overline{C}}{\overline{C}_1} \Big{[}-  \overline A \overline{C}_1 
\overline{T} \Big{(} 2\frac{\overline A_1 \overline{C}_1}{\overline A \overline{C}} \tilde \alpha\tilde u^2 + 
\sum_{i\geq 2}\tilde a_{i}\tilde u \tilde c_{i} \Big{)} + \overline{B}\overline{T} \Big{(} 
2\frac{\overline{B}_2}{\overline{B}} \tilde \beta \tilde c_{2} - \sum_{i\geq 3}\tilde b_{i}\tilde c_{i} \Big{)} 
\Big{]}\,, \\[0.8em]
\ds \frac{d}{d\tilde t} \tilde c_{2} = \ds   \overline A \overline{C}_1 \overline{T} (\frac{\overline A_1 
\overline{C}_1}{\overline A \overline{C}} \tilde \alpha \tilde u^2 -\tilde a_{2}\tilde u \tilde c_{2}) 
-\overline{B}\overline{T} (\frac{\overline{B}_2}{\overline{B}}\tilde \beta \tilde c_{2}-\tilde b_{3}\tilde c_{3})\,, 
\\[0.8em]
\ds \frac{d}{d\tilde t} \tilde c_{i} = \ds   \overline A \overline{C}_1 \overline{T} (\tilde a_{i-1}\tilde u \tilde 
c_{i-1}-\tilde a_{i}\tilde u \tilde c_{i}) -\overline{B}\overline{T} (\tilde b_{i}\tilde c_{i}-\tilde b_{i+1}\tilde 
c_{i+1}) \,, \quad i\geq 3\,\,. \\[0.8em]
\end{array}
\end{equation*}
The mass conservation reads
\begin{equation*}
 \tilde u( \tilde t) + \frac{\overline C}{\overline C_1} \sum_{i\geq 2} i \tilde c_i(\tilde t) = \frac{\overline 
M_c}{\overline C_1} \tilde m  \,.
\end{equation*}
We introduce the scaling parameter $\veps>0$ for the size of the clusters. Namely, a cluster  of size $i$ is now seen 
as a cluster of size roughly $\veps i$ so that we can define the density \eqref{eq:def_feps}. Then, the scaling 
obtained in Eq.~\eqref{sys:BD_rescaled} corresponds to the following choice of relations between the characteristic 
values
\begin{equation*}
 \overline C/ \overline C_1 = \veps^2 \,,\quad   \overline A \, \overline C_1\overline  T = \overline B \, \overline T 
=  \frac{1}{\veps}\,,\quad   \overline M_c / \overline C_1 = 1\,, 
\end{equation*}
and, at the boundary,
\begin{equation*}
 \overline A_1 = \veps^{2} \overline A \,,
\end{equation*}
and
\begin{equation*}
  \overline B_2 = \veps^{\eta}\overline B\,,
\end{equation*} 
with $\eta\geq 0$. The reader interested in a physical justification of this scaling can refer to the discussion in
Section~\ref{sec:disc} and to \cite{Collet2002}. 

 \bibliographystyle{abbrv}
 \bibliography{biblio}

\end{document}